\documentclass[11pt, reqno]{amsart}
\usepackage[utf8]{inputenc}
\usepackage{amsmath, amssymb, amsfonts, latexsym, amsthm}
\usepackage{calligra}
\usepackage{tcolorbox}

\newtheorem{thm}{Theorem}[section]
\newtheorem{cor}[thm]{Corollary}
\newtheorem{lem}[thm]{Lemma}
\newtheorem{prop}[thm]{Proposition}
\newtheorem{rem}[thm]{Remark}

\numberwithin{equation}{section}

\newcommand{\mbb}{\mathbb}

\newcommand{\ov}{\overline}

\newcommand{\Om}{\Omega}

\begin{document}
\title[Weighted Szeg\H{o} kernels]{Weighted Szeg\H{o} kernels on planar domains}
    \thanks{The first named author was supported in part by the PMRF Ph.D. fellowship of the Ministry of Education, Government of India.}
	\author{Aakanksha Jain, Kaushal Verma} 
\address{AJ: Department of Mathematics, Indian Institute of Science, Bengaluru-560012, India}
\email{aakankshaj@iisc.ac.in}

\address{KV: Department of Mathematics, Indian Institute of Science, Bengaluru-560012, India}
\email{kverma@iisc.ac.in}

\begin{abstract}
We study properties of weighted Szeg\H{o} and Garabedian kernels on planar domains. Motivated by the unweighted case as explained in Bell's work, the starting point is a weighted Kerzman--Stein formula that yields boundary smoothness of the weighted Szeg\H{o} kernel. This provides information on the dependence of the weighted Szeg\H{o} kernel as a function of the weight. When the weights are close to the constant function $1$ (which corresponds to the unweighted case), it is shown that some properties of the unweighted Szeg\H{o} kernel propagate to the weighted Szeg\H{o} kernel as well. Finally, it is shown that the reduced Bergman kernel and higher order reduced Bergman kernels can be written as a rational combination of three unweighted Szeg\H{o} kernels and their conjugates, thereby extending Bell's list of kernel functions that are made up of simpler building blocks that involve the Szeg\H{o} kernel.
\end{abstract}

\maketitle

\section{Introduction}

\noindent Bell's work builds a paradigm that reveals the building blocks for several kernel functions associated with finitely connected domains in the complex plane. The meta theorem that emerges from this paradigm is this -- kernel functions, when viewed as functions of two complex variables, can be realized as combinations (rational or algebraic) of finitely many functions of one complex variable and their conjugates. To cite an example, \cite{Be1} shows that the classical Bergman and Szeg\H{o} kernels associated with an $n$-connected domain in the plane with no isolated boundary points can be realized as rational combinations of three holomorphic functions of a single variable and their conjugates. Other manifestations of this phenomenon exist; they have been systematically studied in \cite{Be2, Be3} among others. 

\medskip

The primary motivation for this note goes back to \cite{Be1}. We study various aspects of weighted Szeg\H{o} kernels on planar domains with no isolated boundary point. As lucidly explained in \cite{Be0}, the Kerzman--Stein formula provides an elegant pathway to understand kernel functions in the non-weighted (or equivalently, classical) case. The first thing, therefore, is to write down a weighted version of this formula and this is done in Section $\ref{Kerzman-Stein}$. Section $\ref{Szego_kernel}$ looks at smoothness to the boundary properties of the weighted Szeg\H{o} kernel. While in the classical case, this is a consequence of pseudolocal estimates for the Cauchy transform and is explained in Section $24$ of \cite{Be0}, the same question for the weighted Szeg\H{o} kernel does not seem to have been addressed so far.  A consequence of this is a Ramadanov type theorem for the weighted Szeg\H{o} kernels when viewed as functions of the weights. Of particular interest here is the situation when the family of weights is close to the constant function $\bf 1$, which corresponds to the classical case. The stability results in Section $\ref{weights_constant}$ imply that properties of the classical Szeg\H{o} kernel carry over to the weighted Szeg\H{o} kernel when the weights are sufficiently close to constant function $\bf 1$ -- this validates an expected phenomenon. Finally, in Section $\ref{Reduced_Bergman}$, we look at higher-order reduced Bergman kernels and use the methods of \cite{Be1} to show that they can also be represented as rational combinations of three classical Szeg\H{o} kernels viewed as functions of one complex variable and their conjugates. This extends the list of kernels that possess the property of being expressible in terms of simpler functions.

\medskip

Unless stated otherwise in what follows, $\Omega$ will denote a bounded $n$-connected domain in the plane with $C^{\infty}$ smooth boundary and $\varphi$ a positive real-valued $C^{\infty}$ function on $\partial\Omega$.\footnote{When defining objects with respect to a weight $\varphi$, we shall put $\varphi$ as a subscript. When $\varphi\equiv 1$, the subscript will be dropped.}

\medskip

\noindent\textbf{Acknowledgment.} The authors would like to thank Steven R. Bell for his encouragement, valuable email exchanges, and particularly his help in the proof of Theorem \ref{Bell1} and Corollary \ref{Bell2}. The authors also thank the reviewer for useful inputs that helped improve the article.

\section{The Kerzman--Stein formula with weights}\label{Kerzman-Stein}

Following \cite{Be0}, \cite{Be1}, let $L^2_{\varphi}(\Omega)$ be the Hilbert space of complex-valued functions on $\partial\Omega$ that are square integrable with respect to the arc length measure $\varphi\,ds$. Here, $ds$ is the differential of arc length and is given by  $ds = \vert z'(t) \vert \; dt$ where $z = z(t)$ is a smooth parametrization of $\partial \Omega$. In terms of the complex unit tangent vector function $T(z(t)) = z'(t) / \vert z'(t) \vert$, $dz = z'(t) \; dt = T ds$. The inner product on $L^2_{\varphi}(\partial\Omega)$ is
\[
\langle u,v\rangle_{\varphi}=\int_{\partial\Omega}u\bar{v}\,\varphi\,ds\quad \text{for }u,v\in L^2_{\varphi}(\partial\Omega).
\]
When $\varphi \equiv 1$, this reduces to the standard inner product $\langle u, v \rangle$. Note that $L^2_{\varphi}(\partial\Omega)=L^2(\partial\Omega)$ as sets. Also, $C^{\infty}(\partial\Omega)$ is dense in $L^2_{\varphi}(\partial\Omega)$ with respect to its norm. For $u\in C^{\infty}(\partial\Omega)$, the Cauchy transform of $u$ is 
\[
(\mathcal{C}u)(z)=\frac{1}{2\pi i}\int_{\partial\Omega}\frac{u(\zeta)}{\zeta-z}d\zeta,\quad z\in\Omega
\]
which is holomorphic on $\Omega$. For $a \in \Omega$ and $z \in \partial \Omega$, let $C_a(z)$ be the conjugate of 
\[
\frac{1}{2 \pi i} \frac{T(z)}{z-a}.
\]
Then, $C_a$ is the Cauchy kernel which defines the Cauchy transform $\mathcal C$ in the sense that $\mathcal C u(z) = \langle u, C_z \rangle$. The weighted Cauchy kernel is defined to be $\varphi^{-1} C_a$ and the corresponding weighted Cauchy transform $\mathcal C_{\varphi}u$ satisfies
\[
(\mathcal C_{\varphi} u)(z) = \langle u, \varphi^{-1} C_z \rangle_{\varphi} = \langle u, C_z \rangle = (\mathcal C u)(z)
\]
which shows that $\mathcal C_{\varphi} = \mathcal C$.

\medskip

Let $A^{\infty}(\Omega)$ denote the space of holomorphic functions on $\Omega$ that are in $C^{\infty}(\ov{\Omega})$. The Cauchy transform $\mathcal C$ maps $C^{\infty}(\partial\Omega)$ into $A^{\infty}(\Omega)$  and this allows the Cauchy transform to be viewed as an operator from $C^{\infty}(\partial\Omega)$ into itself.

\medskip

Let $u,v\in C^{\infty}(\partial\Omega)$ be arbitrary. We wish to construct a function $\mathcal{C}^*_{\varphi}v$ in $C^{\infty}(\partial\Omega)$ such that $\langle\mathcal{C}u,v\rangle_{\varphi}=\langle u,\mathcal{C}^{*}_{\varphi}v\rangle_{\varphi}$ for all $u\in C^{\infty}(\partial\Omega)$. It is known that (see \cite{Be0})
\[
\langle \mathcal{C}u,v\rangle=\langle u,v-\ov{T}\overline{\mathcal{C}(\bar{v}\ov{T})}\rangle
\]
for all $u,v\in C^{\infty}(\partial \Omega)$. Therefore,
\[
\langle\mathcal{C}u,v\rangle_{\varphi}
=
\langle \mathcal{C}u, v\varphi\rangle
=
\langle u,v\varphi-\overline{T}\overline{\mathcal{C}(\overline{v\varphi}\overline{T})}\rangle
=
\langle u,v-\overline{T}\varphi^{-1}\overline{\mathcal{C}(\overline{v\varphi}\overline{T})}\rangle_{\varphi}.
\]
Thus, for $v\in C^{\infty}(\partial\Omega)$, define
\begin{equation}
\mathcal{C}^{*}_{\varphi}v=v-\overline{T}\varphi^{-1}\overline{\mathcal{C}(\overline{v\varphi}\overline{T})}
\end{equation}
which shows that $\mathcal{C}^{*}_{\varphi}v\in C^{\infty}(\partial\Omega)$ and $\langle\mathcal{C}u,v\rangle_{\varphi}=\langle u,\mathcal{C}^{*}_{\varphi}v\rangle_{\varphi}$ for all $u,v\in C^{\infty}(\partial\Omega)$.

\medskip 

Let $A^{\infty}(\partial\Omega)$ denote the set of functions on $\partial\Omega$ that are the boundary values of functions in $A^{\infty}(\Omega)$. The Hardy space $H^2(\partial\Omega)$ is defined to be the closure in $L^2(\partial\Omega)$ of $A^{\infty}(\partial\Omega)$ and members of the Hardy space are in one-to-one correspondence with the space of holomorphic functions on $\Omega$ with $L^2$ boundary values. The Hardy space can be identified in a natural way with the subspace of $L^2_{\varphi}(\partial\Omega)$, equal to the closure of $A^{\infty}(\partial\Omega)$ in that space. Thus, there is no need to define $H^2_{\varphi}(\partial\Omega)$ separately. But we shall use the notation $H^2_{\varphi}(\partial\Omega)$ whenever there is a need to emphasize that the Hardy space is endowed with the inner product $\langle \cdot,\cdot\rangle_{\varphi}$. 

\medskip

The orthogonal projection $P_{\varphi}$ from $L^2_{\varphi}(\partial\Omega)$ onto $H^2(\partial\Omega)$ is the weighted Szeg\H{o} projection. The Cauchy Transform and the weighted Szeg\H{o} projection are related by the weighted Kerzman-Stein formula as follows, in which $\mathcal{A}_{\varphi} = \mathcal{C}-\mathcal{C}^{*}_{\varphi}$ is the weighted Kerzman-Stein operator.

\medskip

\begin{prop}\label{kzfor}
$P_{\varphi}(I+\mathcal{A}_{\varphi})=\mathcal{C}$ on $C^{\infty}(\partial\Omega)$,
where $I$ denotes the identity operator.
\end{prop}

\begin{proof}
For $u\in C^{\infty}(\partial\Omega)$, 
\[
\left(I+(\mathcal{C}-\mathcal{C}^*_{\varphi}) \right)u
=
u+\mathcal{C}u- \left(u-\overline{T}\varphi^{-1}\overline{\mathcal{C}(\overline{u\varphi}\overline{T})} \right)
=
\mathcal{C}u+\overline{T}\varphi^{-1}\overline{\mathcal{C}(\overline{u\varphi}\overline{T})}.
\]
Since $u\phi T\in C^{\infty}(\partial\Omega)$, $\mathcal{C}(\overline{u\varphi}\overline{T})\in A^{\infty}(\Omega)$. Therefore, $\overline{T}\overline{\mathcal{C}(\overline{u\varphi}\overline{T})}$ is orthogonal to $H^2(\partial\Omega)$ with respect to the standard inner product. This implies that $\overline{T}\varphi^{-1}\overline{\mathcal{C}(\overline{u\varphi}\overline{T})}$ is orthogonal to $H^2(\partial\Omega)$ with respect to the weighted inner product. Since $\mathcal{C}u\in A^{\infty}(\Omega)\subset H^2(\partial\Omega)$,
\begin{equation}
(P_{\varphi} (I+\mathcal{A}_{\varphi}))(u)=\mathcal{C}u.
\end{equation}
and this completes the proof. 
\end{proof}

\begin{prop}\label{kzkernel} The weighted Kerzman-Stein operator $\mathcal{A}_{\varphi}$ is represented by a kernel $A_{\varphi}(\cdot,\cdot)$ in $C^{\infty}(\partial\Omega\times\partial\Omega)$. That is,
\begin{equation}
(\mathcal{A}_{\varphi}u)(z)=\int_{\zeta\in\partial\Omega} A_{\varphi}(z,\zeta) u(\zeta)\varphi(\zeta)\,ds.
\end{equation}
The kernel is
\begin{equation}
A_{\varphi}(z,\zeta)=\frac{1}{2\pi i}\left(\frac{T(\zeta)\varphi^{-1}(\zeta)}{\zeta-z}-\frac{\overline{T(z)}\varphi^{-1}(z)}{\bar{\zeta}-\bar{z}}\right),\quad z,\zeta\in\partial\Omega
\end{equation}
and will be called the weighted Kerzman-Stein kernel.
\end{prop}

\begin{proof}
To understand $\mathcal{A}_{\varphi}=\mathcal{C}-\mathcal{C}^*_{\varphi}$, for $u\in C^{\infty}(\partial\Omega)$ and $z_0\in\partial\Omega$, let
\[
(\mathcal{H}u)(z_0) = \textbf{P.V.}\,\frac{1}{2\pi i}\int_{\partial\Omega} \frac{u(\zeta)}{\zeta-z_0}\,d\zeta.
\]
$\mathcal{H}u$ is well defined and Plemelj's theorem shows that
\[
\mathcal{C}u(z_0)=\frac{1}{2}u(z_0) + (\mathcal{H}u)(z_0).
\]
Therefore, for $u\in C^{\infty}(\partial\Omega)$,
\begin{eqnarray*}
\mathcal{A}_{\varphi}u &=& \mathcal{C}u-u+\overline{T}\varphi^{-1}\overline{\mathcal{C}(\overline{u\varphi}\overline{T})}
\\
&=& \frac{1}{2}u + \mathcal{H}u-u+\overline{T}\varphi^{-1}\left(\overline{\frac{1}{2}\overline{u\varphi}\overline{T}+\mathcal{H}(\overline{u\varphi}\overline{T})}\right)
\\
&=& \mathcal{H}u + \overline{T}\varphi^{-1}\overline{\mathcal{H}(\overline{u\varphi}\overline{T})}.
\end{eqnarray*}
So, for $z\in\partial\Omega$, we have
\begin{eqnarray*}
(\mathcal{A}_{\varphi}u)(z) &=& \textbf{P.V.}\,\frac{1}{2\pi i}\int_{\partial\Omega} \frac{u(\zeta)}{\zeta-z}\,d\zeta + \overline{T(z)}\varphi^{-1}(z)\left(\overline{\textbf{P.V.}\,\frac{1}{2\pi i}\int_{\partial\Omega} \frac{(\overline{u\varphi T})(\zeta)}{\zeta-z}\,d\zeta}\right)
\\
&=& \textbf{P.V.}\,\frac{1}{2\pi i}\int_{\zeta\in\partial\Omega} \frac{u(\zeta)T(\zeta)}{\zeta-z}\,ds - \overline{T(z)}\varphi^{-1}(z) \left(\textbf{P.V.}\,\frac{1}{2\pi i}\int_{\zeta\in\partial\Omega} \frac{u(\zeta)\varphi(\zeta)}{\bar{\zeta}-\bar{z}}\,ds\right)
\\
&=& \textbf{P.V.}\,\frac{1}{2\pi i}\int_{\zeta\in\partial\Omega}\left[\frac{T(\zeta)\varphi^{-1}(\zeta)}{\zeta-z}-\frac{\overline{T(z)}\varphi^{-1}(z)}{\bar{\zeta}-\bar{z}}\right]u(\zeta)\varphi(\zeta)\,ds.
\end{eqnarray*}
For $z,\zeta\in\partial\Omega$, let
\[
A_{\varphi}(z,\zeta)=\frac{1}{2\pi i}\left(\frac{T(\zeta)\varphi^{-1}(\zeta)}{\zeta-z}-\frac{\overline{T(z)}\varphi^{-1}(z)}{\bar{\zeta}-\bar{z}}\right).
\]
The $C^{\infty}$-smoothness of $A_{\varphi}(z,\zeta)$ follows from a straightforward adaptation of the reasoning in Section $5$ of \cite{Be0} that deals with the unweighted version of the Kerzman--Stein kernel. Since $A_{\varphi}\in C^{\infty}(\partial\Omega\times\partial\Omega)$, the principal value integral above is a standard integral. Therefore,
\[
(\mathcal{A}_{\varphi}u)(z)=\int_{\zeta\in\partial\Omega} A_{\varphi}(z,\zeta) u(\zeta)\varphi(\zeta)\,ds.
\]
\end{proof}

\begin{thm}
The Cauchy transform $\mathcal{C}_{\varphi} \;(\text {which equals} \;\mathcal C)$ extends to a bounded operator from $L^2_{\varphi}(\partial\Omega)$ onto $H^2(\partial\Omega)$. Hence,
\begin{enumerate}
    \item $\mathcal{C}^{*}_{\varphi}$ extends to be a bounded operator from $L^2_{\varphi}(\partial\Omega)$ to itself.
    \item The identity $\langle \mathcal{C}_{\varphi}u,v\rangle_{\varphi}=\langle u,\mathcal{C}^{*}_{\varphi}v\rangle_{\varphi}$ holds for all $u,v\in L^2_{\varphi}(\partial\Omega)$. Therefore, $\mathcal{C}^{*}_{\varphi}$ is the $L^2_{\varphi}$ adjoint of $\mathcal{C}_{\varphi}$.
    \item The Kerzman-Stein formula $P_{\varphi}(I+\mathcal{A}_{\varphi})=\mathcal{C}_{\varphi}$ holds on $L^2_{\varphi}(\partial\Omega)$.
\end{enumerate}
\end{thm}
\begin{proof}
It follows from Proposition \ref{kzkernel} that $\mathcal{A}_{\varphi}$ maps $L^2_{\varphi}(\partial\Omega)$ into $C^{\infty}(\partial\Omega)$ and satisfies an $L^2_{\varphi}$ estimate, namely $\lVert \mathcal{A}_{\varphi}u\rVert_{\varphi}\leq c\lVert u\rVert_{\varphi}$. Thus, by Proposition \ref{kzfor} 
\[
\Vert \mathcal{C}u\rVert_{\varphi}\leq (1+c)\lVert u\rVert_{\varphi}
\]
for $u\in C^{\infty}(\partial\Omega)$. Since $C^{\infty}(\partial\Omega)$ is dense in $L^2_{\varphi}(\partial\Omega)$ with respect to its norm, and since $\mathcal{C}$ maps $C^{\infty}(\partial\Omega)$ into $A^{\infty}(\partial\Omega)$, the Cauchy transform extends to be a bounded operator $\mathcal{C}_{\varphi}$ from $L^2_{\varphi}(\partial\Omega)$ into $H^2(\partial\Omega)$. It is known that $\mathcal{C}h = h$ for all $h\in A^{\infty}(\partial\Omega)$. Therefore, $\mathcal{C}_{\varphi}$ is a bounded operator from $L^2_{\varphi}(\partial\Omega)$ onto $H^2_{\varphi}(\partial\Omega)$.
\end{proof}

\begin{rem} 
For $u\in L^2(\partial\Omega)$, both $\mathcal{C}u$ and $P_{\varphi}u$ are in $H^2_{\varphi}(\partial\Omega)$. But $\mathcal{C}u$ need not be the orthogonal projection of $u$ onto $H^{2}_{\varphi}(\partial\Omega)$. The weighted Kerzman-Stein formula gives the relationship between these two operators.
\end{rem}

\section{The weighted Szeg\H{o} kernel}\label{Szego_kernel}

For $a\in\Omega$, the evaluation functional $h\mapsto h(a)$ on $H^2_{\varphi}(\Omega)$ is continuous since
\begin{equation}\label{szego}
\vert h(a)\vert = \vert \langle h, C_a \rangle \vert = \vert \langle h, C_a \varphi^{-1}\rangle_{\varphi}\vert \leq \Vert C_a \varphi^{-1} \Vert_{\varphi} \Vert h\Vert_{\varphi}
\end{equation}
and hence, there exists a unique function $S_{\varphi}(\cdot,a)\in H^2(\partial\Omega)$ such that 
\[
h(a)=\langle h, S_{\varphi}(\cdot, a)\rangle_{\varphi}
\]
for all $h\in H^2(\partial\Omega)$. The function $S_{\varphi}(\cdot,\cdot)$ is the weighted Szeg\H{o} kernel of $\Omega$ with respect to the weight $\varphi$. 

\medskip

It can be seen that the weighted Szeg\H{o} kernel is hermitian symmetric, that is, for $z, w\in\Omega$, we have
$
S_{\varphi}(z,w)= \overline{S_{\varphi}(w,z)}.
$
Therefore, $S_{\varphi}(\cdot,\cdot)\in C^{\infty}(\Omega\times\Omega)$. 

\medskip

The weighted Szeg\H{o} kernel is the kernel of the weighted Szeg\H{o} projection because for all $u\in L^2_{\varphi}(\partial\Omega)$ and $a\in\Omega$
\begin{equation}
(P_{\varphi}u)(a) = \langle P_{\varphi}u, S_{\varphi}(\cdot, a)\rangle_{\varphi}=\langle u, S_{\varphi}(\cdot, a)\rangle_{\varphi} = \int_{\zeta\in\partial\Omega}S_{\varphi}(a,\zeta)u(\zeta)\varphi(\zeta) ds.
\end{equation}
Note that for all $h\in H^2(\partial\Omega)$
\[
h(a) = (\mathcal{C}h)(a) = \langle h, C_a\rangle = \langle h, \varphi^{-1}C_a\rangle_{\varphi} =  \langle h, P_{\varphi}(\varphi^{-1}C_a)\rangle_{\varphi}.
\]
Therefore, the weighted Szeg\H{o} kernel $S_{\varphi}(\cdot,a)$ is also given by the weighted Szeg\H{o} projection of $\varphi^{-1}C_a$.
That is,
\begin{equation}
S_{\varphi}(z,a)=(P_{\varphi} (\varphi^{-1}C_a))(z)=\int_{\zeta\in\partial\Omega}S_{\varphi}(z,\zeta)C_a(\zeta) ds.
\end{equation}
for every $z\in\Omega$.
Since the Cauchy transform maps $C^{\infty}(\partial\Omega)$ into itself, it follows from the weighted Kerzman--Stein formula that the weighted Szeg\H{o} projection $P_{\varphi}$ also maps $C^{\infty}(\partial\Omega)$ into itself. Thus, $S_{\varphi}(\cdot,a)\in A^{\infty}(\Omega)$.

\medskip

A function $u$ in $L^2(\partial\Omega)$ is orthogonal to $H^2_{\varphi}(\partial\Omega)$ if and only if $u\varphi$ is orthogonal to $H^2(\partial\Omega)$. All the functions $v\in L^2(\partial\Omega)$ orthogonal to $H^2(\partial\Omega)$ are of the form $\overline{HT}$ where $H\in H^2(\partial\Omega)$. If $v\in C^{\infty}(\partial\Omega)$ then $H\in A^{\infty}(\Omega)$. So, the orthogonal decomposition of $\varphi^{-1}C_a$ is given by
\begin{equation}
\varphi^{-1}C_a = S_{\varphi}(\cdot,a) + \varphi^{-1} \overline{H_aT}
\end{equation}
where $H_a\in A^{\infty}(\Omega)$. Also, the above decomposition shows that $H_a$ is holomorphic in $a\in\Omega$ for fixed $z\in\partial\Omega$. The weighted Garabedian kernel of $\Omega$ with respect to $\varphi$ is defined by
\begin{equation}
L_{\varphi}(z,a) = \frac{1}{2\pi}\frac{1}{z-a} - i H_a(z).
\end{equation}
For a fixed $a\in\Omega$, $L_{\varphi}(z,a)$ is a holomorphic function of $z$ on $\Omega\setminus\{a\}$ with a simple pole at $z=a$ with residue $\frac{1}{2\pi}$, and extends $C^{\infty}$ smoothly to $\partial\Omega$. Further, $L_{\varphi}(z,a)$ is holomorphic in $a$ on $\Omega$ for fixed $z\in \partial\Omega$. Moreover, it is known that (see \cite{Nehari})
\begin{equation}\label{garabedian}
L_{\varphi}(z,a) = - L_{1/\varphi}(a,z) \quad  z,a\in\Omega.
\end{equation}
Therefore, for a fixed $z\in\Omega$, $L_{\varphi}(z,a)$ is a holomorphic function of $a$ on $\Omega\setminus\{z\}$ with a simple pole at $a=z$ and residue $\frac{1}{2\pi}$, and extends $C^{\infty}$ smoothly to $\partial\Omega$. Finally, for $z \in \partial \Omega$, $a \in \Omega$
\[
S_{\varphi}(a,z) = \overline{S_{\varphi}(a,z)} 
=
\frac{1}{\varphi(z)}\left(\frac{1}{2\pi i}\frac{T(z)}{z-a}-H_a(z) T(z)\right)
=
\frac{1}{i\varphi(z)}\left(\frac{1}{2\pi}\frac{1}{z-a}-iH_a(z)\right)T(z)
\]
shows that the weighted Szeg\H{o} kernel and the weighted Garabedian kernel satisfy the identity
\begin{equation}
S_{\varphi}(a,z) = \frac{1}{i\varphi(z)} L_{\varphi}(z,a) T(z).
\end{equation}

\medskip

Let $z(t)$ denote the parametrization of the boundary $\partial\Omega$ where $t$ ranges over the parameter interval $J$. For a non-negative integer $s$, define the norm $\Vert u\Vert_s$ of a function $u$ defined on the boundary of $\Omega$ as
\[
\Vert u\Vert_s = \sup\left\lbrace\left\lvert \frac{d^m}{dt^m}u(z(t))\right\rvert : t\in J, \, \, 0\leq m\leq s\right\rbrace.
\]
Let $C^s(\partial\Omega) = \{ u : \Vert u\Vert_s < \infty\}$. This space does not depend upon the parametrization but the norm does.

\medskip

Theorem $9.2$ in \cite{Be0} shows that for a given a non-negative integer $s$, there is a positive integer $n=n(s)$ and a constant $K=K(s)$ such that
    \[
    \Vert \mathcal{C}u\Vert_s\leq K\Vert u\Vert_n
    \quad\quad
    \text{and}
    \quad\quad
    \Vert Pu\Vert_s \leq K \Vert u\Vert_n
    \]
for all $u\in C^{\infty}(\partial\Omega)$. Consequently, since $C^{\infty}(\partial\Omega)$ is dense in $C^n(\partial\Omega)$, the same inequalities hold for all $u\in C^n(\partial\Omega)$. In particular, it follows that $Pu$ and $\mathcal{C}u$ are in $C^s(\partial\Omega)$ whenever $u\in C^n(\partial\Omega)$. The weighted analogs of these estimates are essential in understanding the boundary smoothness of $S_{\varphi}(z, a)$.

\begin{thm}[\textbf{Estimates}]
Let $s$ be a non-negative integer. Then there exists a positive integer $n=n(s)$ and a constant $C=C(s,\varphi)$ such that
\begin{equation}
\Vert P_{\varphi}u\Vert_s \leq C\Vert u\Vert_n
\end{equation}
for all $u\in C^{\infty}(\partial\Omega)$. Consequently, since $C^{\infty}(\partial\Omega)$ is dense in $C^n(\partial\Omega)$, the same inequality holds for all $u\in C^n(\partial\Omega)$. In particular, it follows that $P_{\varphi}u$ is in $C^s(\partial\Omega)$ whenever $u\in C^n(\partial\Omega)$.
\end{thm}

\begin{proof}
Let $s$ be a non-negative integer. Then there exists a positive integer $n=n(s)$ and a constant $K=K(s)$ such that
\[
\Vert \mathcal{C}u\Vert_s\leq K\Vert u\Vert_n
\]
for all $u\in L^2(\partial \Omega)$. The weighted Kerzman-Stein identity states that $P_{\varphi}(I+\mathcal{A}_{\varphi})=\mathcal{C}_{\varphi}$. Taking the $L^2_{\varphi}$ adjoint on both sides and using the fact that $\mathcal{A}_{\varphi}^{*}=(\mathcal{C}_{\varphi}-\mathcal{C}^*_{\varphi})^{*}=\mathcal{C}^*_{\varphi}-\mathcal{C}_{\varphi}=-\mathcal{A}_{\varphi}$ gives 
\[
(I-\mathcal{A}_{\varphi})P_{\varphi}=\mathcal{C}^*_{\varphi}.
\]
On subtracting the above formula from the weighted Kerzman-Stein formula, we get $P_{\varphi}\mathcal{A}_{\varphi}+\mathcal{A}_{\varphi}P_{\varphi}=\mathcal{A}_{\varphi}$ and hence $P_{\varphi}\mathcal{A}_{\varphi}=\mathcal{A}_{\varphi}(I-P_{\varphi})$. Since $P_{\varphi}=\mathcal{C}-P_{\varphi}\mathcal{A}_{\varphi}$, we obtain that
\begin{equation}\label{KS eqn}
    P_{\varphi}=\mathcal{C}-\mathcal{A}_{\varphi}(I-P_{\varphi}).
\end{equation}
We shall use this formula to give estimates for $P_{\varphi}$.
Let $z(t)$ denote the parametrization of $\partial\Omega$ where $t$ ranges over the domain of parametrization $J$. Then, for $u\in L^2_{\varphi}(\partial\Omega)$
\begin{eqnarray*}
\Vert \mathcal{A}_{\varphi}u\Vert_s
&=&
\sup\left\lbrace\left\lvert\frac{d^m}{dt^m}(\mathcal{A}_{\varphi}u)(z(t))\right\rvert : t\in J,\,0\leq m\leq s\right\rbrace\\
&=&
\sup\left\lbrace\left\lvert \frac{d^m}{dt^m}\left(\int_{\zeta\in\partial\Omega} A_{\varphi}(z(t),\zeta) u(\zeta)\varphi(\zeta)\,ds\right)\right\rvert: t\in J,\,0\leq m\leq s\right\rbrace\\
&=&
\sup\left\lbrace\left\lvert\int_{\zeta\in\partial\Omega} \left(\frac{d^m}{dt^m} A_{\varphi}(z(t),\zeta)\right) u(\zeta)\varphi(\zeta)\,ds\right\rvert:t\in J,\,0\leq m\leq s\right\rbrace\\
&\leq&
\int_{\zeta\in\partial\Omega} \sup\left\lbrace\left\lvert \frac{d^m}{dt^m} A_{\varphi}(z(t),\zeta) \right\rvert:t\in J,\,0\leq m\leq s\right\rbrace \vert u(\zeta)\vert\varphi(\zeta)\,ds\\
&\leq&
\sup\left\lbrace\left\lvert \frac{d^m}{dt^m} A_{\varphi}(z(t),\zeta) \right\rvert:t\in J,\,0\leq m\leq s,\,\zeta\in\partial\Omega\right\rbrace\int_{\zeta\in\partial\Omega}\vert u(\zeta)\vert \varphi(\zeta)\,ds\\
&\leq&
C_1\sqrt{\int_{\zeta\in\partial\Omega}\vert u(\zeta)\vert^2\varphi(\zeta)\,ds}
=
C_1\Vert u\Vert_{L^2_{\varphi}(\partial\Omega)},
\end{eqnarray*}
where
\[
C_1
=
C_1(s,\varphi)
=
\sup\left\lbrace\left\lvert \frac{d^m}{dt^m} A_{\varphi}(z(t),\zeta) \right\rvert:t\in J,\,0\leq m\leq s,\,\zeta\in\partial\Omega\right\rbrace 
\sqrt{\int_{\zeta\in\partial\Omega}\varphi(\zeta)\,ds}.
\]
Finally, (\ref{KS eqn}) shows that 
\begin{eqnarray*}
\Vert P_{\varphi}u\Vert_s
&=&
\Vert \mathcal{C}u -\mathcal{A}_{\varphi}(I-P_{\varphi})u\Vert_s
\leq 
\Vert\mathcal{C}u\Vert_s + \Vert \mathcal{A}_{\varphi}(I-P_{\varphi})u\Vert_s
\\
&\leq&
K\Vert u\Vert_n +C_1\Vert (I-P_{\varphi})u\Vert_{L^2_{\varphi}(\partial\Omega)}
\leq
K\Vert u\Vert_n +C_1\Vert u\Vert_{L^2_{\varphi}(\partial\Omega)}\\
&\leq&
K\Vert u\Vert_n +C_1 C_2\Vert u\Vert_0,
\end{eqnarray*}
where $C_2=\sqrt{\left(\int_{\zeta\in\partial\Omega}\varphi(\zeta)\right)}$. Therefore, we have proved that $\Vert P_{\varphi}u\Vert_s\leq C\Vert u\Vert_n$, where
\begin{equation}
C=C(s,\varphi)
=
K+
\sup\left\lbrace\left\lvert \frac{d^m}{dt^m} A_{\varphi}(z(t),\zeta) \right\rvert:t\in J,\,0\leq m\leq s,\,\zeta\in\partial\Omega\right\rbrace 
\int_{\zeta\in\partial\Omega}\varphi(\zeta)\,ds.
\end{equation}
The other conclusions are straightforward.
\end{proof}

\begin{thm}\label{smooth}
Let $\Omega$ be a bounded $n$-connected domain with $C^{\infty}$ smooth boundary and $\varphi$ be a positive function in $C^{\infty}(\partial\Omega)$. Then, $S_{\varphi}(z,w)\in C^{\infty}((\overline{\Omega}\times\overline{\Omega})-\Delta)$, where $\Delta=\{(z,z):z\in\partial\Omega\}$ is the diagonal boundary set.
\end{thm}
\begin{proof}
Using the weighted Kerzman-Stein identity $P_{\varphi}(I+\mathcal{A}_{\varphi})=\mathcal{C}$ for functions in $C^{\infty}(\partial\Omega)$, we can write for $z,w\in\Omega$ that
\[
S_{\varphi}(z,w)
=
(P_{\varphi} C_{\varphi,w})(z)
=
(\mathcal{C} C_{\varphi,w})(z) - (P_{\varphi} \mathcal{A}_{\varphi} C_{\varphi,w})(z)
=
H_1(z,w)-H_2(z,w).
\]
We will first study $H_1(z,w)$. For $z,w\in\Omega$
\begin{eqnarray*}
H_1(z,w)
&=&
(\mathcal{C}C_{\varphi,w})(z)=\frac{1}{2\pi i}\int_{\partial\Omega} \frac{C_{\varphi,w}(\xi)}{\xi-z} d\xi=\frac{1}{2\pi i}\int_{\partial\Omega} \overline{\left(\frac{1}{2\pi i}\frac{\varphi^{-1}(\xi)T(\xi)}{\xi-w}\right)}\frac{1}{\xi-z} d\xi\\
&=&
\frac{1}{4\pi^2}\int_{\xi\in\partial\Omega}\frac{\varphi^{-1}(\xi)}{(\xi-z)(\overline{\xi}-\overline{w})} ds.
\end{eqnarray*}
So, $H_1(z,w)\in C^{\infty}(\Omega\times\Omega)$. Let $z_0, w_0\in\partial\Omega$ be such that $z_0\neq w_0$. Choose $\epsilon>0$ such that $\overline{D_{\epsilon}(z_0)}\cap \overline{D_{\epsilon}(w_0)}=\emptyset$. Choose $\chi\in C^{\infty}(\partial\Omega)$ such that $\chi\equiv 1$ on $D_{\epsilon}(z_0)\cap\partial\Omega$ and $\chi\equiv 0$ on $D_{\epsilon}(w_0)\cap\partial\Omega$. For $z,w\in\Omega$
\[
H_1(z,w)=(\mathcal{C} C_{\varphi,w})(z) 
=
(\mathcal{C}(\chi C_{\varphi,w}))(z)+(\mathcal{C}((1-\chi)C_{\varphi,w}))(z).
\]
For $w\in D_{\epsilon}(w_0)\cap\overline{\Omega}$, the function $\chi C_{\varphi,w}$ is in $C^{\infty}(\partial\Omega)$. Therefore, $(\mathcal{C}(\chi C_{\varphi,w}))(z)\in C^{\infty}(\overline{\Omega})$ as a function of $z$ for every $w\in D_{\epsilon}(w_0)\cap\overline{\Omega}$. Given a non-negative integer $s$, there exists a positive integer $n=n(s)$ and a constant $K=K(s)>0$ such that $\Vert \mathcal{C}u\Vert_{s}\leq K\Vert u\Vert_{n}$ for all $u\in C^{\infty}(\partial\Omega)$. 

\medskip

Fix $\tilde{z}\in\overline{\Omega}$ and  $\tilde{w}\in D_{\epsilon}(w_0)\cap \overline{\Omega}$.
Given $\epsilon>0$, we have $\Vert\chi C_{\varphi,w}-\chi C_{\varphi,\tilde{w}}\Vert_n<\epsilon$ and therefore $\Vert\mathcal{C}(\chi C_{\varphi,w})-\mathcal{C}(\chi C_{\varphi,\tilde{w}})\Vert_s<K\epsilon$ for all $w$ close enough to $\tilde{w}$. Hence, for $z$ and $w$ close enough to $\tilde{z}$ and $\tilde{w}$ respectively, we have
\begin{multline*}
\vert (\mathcal{C}(\chi C_{\varphi,w}))(z)-(\mathcal{C}(\chi C_{\varphi,\tilde{w}}))(\tilde{z})\vert 
\leq
\vert (\mathcal{C}(\chi C_{\varphi,w}))(z)-(\mathcal{C}(\chi C_{\varphi,\tilde{w}}))(z)\vert
+
\\
\vert (\mathcal{C}(\chi C_{\varphi,\tilde{w}}))(z)-(\mathcal{C}(\chi C_{\varphi,\tilde{w}}))(\tilde{z})\vert
< K\epsilon +\epsilon.
\end{multline*}
Here, the first term is less than $K\epsilon$ for $w$ close enough to $\tilde{w}$ (and all $z\in\overline{\Omega}$), and the second term is less than $\epsilon$ for $z$ close enough to $\tilde{z}$ (since $\mathcal{C}(\chi C_{\varphi,\tilde{w}})\in C^{\infty}(\overline{\Omega})$). Repeat the last step for the derivatives of $(\mathcal{C}(\chi C_{\varphi,w}))(z)$ with respect to $z$ up to order $s$. Since $s$ is an arbitrary non-negative integer, we have shown that $(\mathcal{C}(\chi C_{\varphi,w}))(z)$ and all its derivatives with respect to $z$ extend continuously to $\overline{\Omega}\times (D_{\epsilon}(w_0)\cap\overline{\Omega})$. Similarly, we can show that for any non-negative integer $k$, the function
\[
\frac{\partial^k}{\partial w^k}(\mathcal{C}(\chi C_{\varphi,w}))(z)=\left(\mathcal{C}\left(\chi \frac{\partial^k}{\partial w^k}C_{\varphi,w}\right)\right)(z)
\]
and all its derivative with respect to $z$ extend continuously to $\overline{\Omega}\times (D_{\epsilon}(w_0)\cap\overline{\Omega})$. Thus, we have proved that $(\mathcal{C}(\chi C_{\varphi,w}))(z)\in C^{\infty}(\overline{\Omega}\times (D_{\epsilon}(w_0)\cap\overline{\Omega}))$ as a function of $(z,w)$. 

\medskip 

Now, observe that for $z,w\in\Omega$, 
\begin{eqnarray*}
\mathcal{C}((1-\chi)C_{\varphi,w})(z)
&=&
\frac{1}{2\pi i}\int_{\zeta\in\partial\Omega}(1-\chi)(\zeta) \overline{\frac{1}{2\pi i}\frac{\varphi^{-1}(\zeta)T(\zeta)}{\zeta-w}}\frac{1}{\zeta-z}\,d\zeta
\\
&=&
\overline{\frac{1}{2\pi i}\int_{\zeta\in\partial\Omega}(1-\overline{\chi})(\zeta) \overline{\frac{1}{2\pi i}\frac{\varphi^{-1}(\zeta)}{\zeta-z}}\frac{T(\zeta)}{\zeta-w}\,(\overline{T(\zeta)})^2 d\zeta}\\
&=&
\overline{\frac{1}{2\pi i}\int_{\zeta\in\partial\Omega}(1-\overline{\chi})(\zeta) \overline{\frac{1}{2\pi i}\frac{\varphi^{-1}(\zeta)T(\zeta)}{\zeta-z}}\frac{1}{\zeta-w}\,d\zeta}
=
\overline{\mathcal{C}((1-\overline{\chi})C_{\varphi,z})(w)}.
\end{eqnarray*}
Since $1-\overline{\chi}\equiv 0$ on $D_{\epsilon}(z_0)\cap\partial\Omega$, we observe by the same arguments as above that the function $(\mathcal{C}((1-\overline{\chi})C_{\varphi,z}))(w)\in C^{\infty}((D_{\epsilon}(z_0)\cap\overline{\Omega})\times \overline{\Omega})$ as a function of $(z,w)$. 

\medskip

Hence, $H_1(z,w)\in C^{\infty}((D_{\epsilon}(z_0)\cap\overline{\Omega})\times (D_{\epsilon}(w_0)\cap\overline{\Omega}))$. Since $(z_0,w_0)\in (\partial\Omega\times\partial\Omega)-\Delta$ was arbitrary, $H_1(z,w)\in C^{\infty}((\overline{\Omega}\times\overline{\Omega})-\Delta)$.

\medskip

We will now study $H_2(z,w)=(P_{\varphi} \mathcal{A}_{\varphi} C_{\varphi,w})(z)$. For $w\in\Omega$ and $\zeta\in\partial\Omega$,
\begin{eqnarray*}
(\mathcal{A}_{\varphi}C_{\varphi,w})(\zeta)&=&\int_{\xi\in\partial\Omega} A_{\varphi}(\zeta,\xi) \, C_{\varphi,w}(\xi) \,ds
=
\frac{-1}{2\pi i}\int_{\xi\in\partial\Omega} A_{\varphi}(\zeta,\xi) \overline{\left(\frac{T(\xi)}{\xi-w}\right)} ds\\
&=&
\overline{\frac{1}{2\pi i}\int_{\xi\in\partial\Omega} \frac{\overline{A_{\varphi}(\zeta,\xi)}}{\xi-w} d\xi}
=
\overline{(\mathcal{C} \psi_{\zeta})(w)},
\end{eqnarray*}
where $\psi_{\zeta}(\xi)=\overline{A_{\varphi}(\zeta,\xi)}$. Since $\psi_{\zeta}\in C^{\infty}(\partial\Omega)$, the function $\mathcal{C}\psi_{\zeta}\in A^{\infty}(\Omega)$. Therefore, we have that $(\mathcal{A}_{\varphi}C_{\varphi,w})(\zeta)\in C^{\infty}(\overline{\Omega})$ as a function of $w$ for every $\zeta\in\partial\Omega$. Given a non-negative integer $s$, there exist a positive integer $n=n(s)$ and a constant $K=K(s)>0$ such that $\Vert \mathcal{C}u\Vert_{s}\leq K\Vert u\Vert_n$ for all $u\in C^{\infty}(\partial\Omega)$. 

\medskip

Fix $w_0\in\overline{\Omega}$ and $\zeta_0\in\partial\Omega$. Given $\epsilon>0$, we have $\Vert \psi_{\zeta}-\psi_{\zeta_0}\Vert_n<\epsilon$ for $\zeta$ close enough to $\zeta_0$. Therefore, considering as a function of $w$, we have
\[
\Vert (\mathcal{A}_{\varphi} C_{\varphi,w})(\zeta)-(\mathcal{A}_{\varphi}C_{\varphi,w})(\zeta_0)\Vert_s
=
\Vert \overline{\mathcal{C} \psi_{\zeta}}-\overline{\mathcal{C} \psi_{\zeta_0}}\Vert_s
<
K\epsilon
\]
for $\zeta$ close enough to $\zeta_0$. Thus, for $\zeta$ and $w$ close enough to $\zeta_0$ and $w_0$ respectively,
\begin{multline*}
\vert (\mathcal{A}_{\varphi} C_{\varphi,w})(\zeta)-(\mathcal{A}_{\varphi} C_{\varphi,w_0})(\zeta_0)\vert 
\leq 
\vert (\mathcal{A}_{\varphi} C_{\varphi,w})(\zeta) - (\mathcal{A}_{\varphi} C_{\varphi,w})(\zeta_0)\vert
+\\
\vert (\mathcal{A}_{\varphi} C_{\varphi,w})(\zeta_0)-(\mathcal{A}_{\varphi} C_{\varphi,w_0})(\zeta_0)\vert
<
K\epsilon+\epsilon.
\end{multline*}
Here, the first term is less than $K\epsilon$ for $\zeta$ close enough to $\zeta_0$ (and all $w\in\overline{\Omega}$), and the second term is less than $\epsilon$ for $w$ close enough to $w_0$ (since $(\mathcal{A}_{\varphi}C_{\varphi,w})(\zeta_0)\in C^{\infty}(\overline{\Omega})$ as a function of $w$). Repeat the last step for the derivatives of $(\mathcal{A}_{\varphi} C_{\varphi,w})(\zeta)$ with respect to $w$ up to order $s$. Since $s$ is an arbitrary non-negative integer, we have shown that $(\mathcal{A}_{\varphi} C_{\varphi,w})(\zeta)$ and all its derivatives with respect to $w$ are continuous on $\partial\Omega\times\overline{\Omega}$. Let $\zeta(t)$ denote the parametrization of $\partial\Omega$. For a non-negative integer $k$, we have
\[
\frac{\partial^k}{\partial t^k}(\mathcal{A}_{\varphi} C_{\varphi,w})(\zeta(t))=\int_{\xi\in\partial\Omega} \frac{\partial^k}{\partial t^k} A_{\varphi}(\zeta(t),\xi) C_{\varphi,w}(\xi)\,ds=\overline{(\mathcal{C}\psi_{\zeta}^k)(w)},
\]
where $\psi_{\zeta}^k(\xi)=\overline{\frac{\partial^k}{\partial t^k}A_{\varphi}(\zeta(t),\xi)}$. Proceeding as before, we can show that for a non-negative integer $k$, the function $\frac{\partial^k}{\partial t^k}(\mathcal{A}_{\varphi} C_{\varphi,w})(\zeta(t))$ and all its derivatives with respect to $w$ are continuous on $\partial\Omega\times\overline{\Omega}$. Thus, we have proved that $(\mathcal{A}_{\varphi}C_{\varphi,w})(\zeta)\in C^{\infty}(\partial\Omega\times\overline{\Omega})$ as a function of $(\zeta,w)$.

\medskip

For every $w\in\overline{\Omega}$, the function $(\mathcal{A}_{\varphi}C_{\varphi,w})(\zeta)\in C^{\infty}(\partial\Omega)$ and therefore $P_{\varphi}\mathcal{A}_{\varphi}C_{\varphi,w}\in A^{\infty}(\Omega)$, in particular, $P_{\varphi}\mathcal{A}_{\varphi}C_{\varphi,w}\in C^{\infty}(\overline{\Omega})$. Given a non-negative integer $s$, there exist a positive integer $n=n(s)$ and a constant $C=C(s,\varphi)>0$ such that $\Vert P_{\varphi}u\Vert_s\leq C \Vert u\Vert_n$ for all $u\in C^{\infty}(\partial\Omega)$.

\medskip

Fix $z_0,\,w_0\in\overline{\Omega}$. Let $\epsilon$ be arbitary. Since $(\mathcal{A}_{\varphi}C_{\varphi,w})(\zeta)\in C^{\infty}(\partial\Omega\times\overline{\Omega})$ as a function of $(\zeta,w)$, we have $\Vert \mathcal{A}_{\varphi}C_{\varphi,w}-\mathcal{A}_{\varphi}C_{\varphi,w_0}\Vert_n<\epsilon$ and therefore $\Vert P_{\varphi}\mathcal{A}_{\varphi}C_{\varphi,w}-P_{\varphi}\mathcal{A}_{\varphi}C_{\varphi,w_0}\Vert_s<\epsilon$ for $w$ close enough to $w_0$. Thus, for $z$ and $w$ close enough to $z_0$ and $w_0$ respectively,
\begin{multline*}
\vert (P_{\varphi}\mathcal{A}_{\varphi} C_{\varphi,w})(z)-(P_{\varphi}\mathcal{A}_{\varphi} C_{\varphi,w_0})(z_0)\vert 
\leq 
\vert (P_{\varphi}\mathcal{A}_{\varphi} C_{\varphi,w})(z) - (P_{\varphi}\mathcal{A}_{\varphi} C_{\varphi,w_0})(z)\vert
+\\
\vert (P_{\varphi}\mathcal{A}_{\varphi} C_{\varphi,w_0})(z)-(P_{\varphi}\mathcal{A}_{\varphi} C_{\varphi,w_0})(z_0)\vert
<
K\epsilon+\epsilon.
\end{multline*}
Here, the first term is less than $K\epsilon$ for $w$ close enough to $w_0$ (and all $z\in\overline{\Omega}$), and the second term is less than $\epsilon$ for $z$ close enough to $z_0$ (since $P_{\varphi}\mathcal{A}_{\varphi}C_{\varphi,w_0}\in C^{\infty}(\overline{\Omega})$). Repeat the last step for the derivatives of $(P_{\varphi}\mathcal{A}_{\varphi} C_{\varphi,w})(z)$ with respect to $z$ up to order $s$. Since $s$ is an arbitrary non-negative integer, we have shown that $(P_{\varphi}\mathcal{A}_{\varphi} C_{\varphi,w})(z)$ and all its derivatives with respect to $z$ are continuous on $\overline{\Omega}\times\overline{\Omega}$. 

\medskip

Since $(\mathcal{A}_{\varphi}C_{\varphi,w})(\zeta)\in C^{\infty}(\partial\Omega\times\overline{\Omega})$ as a function of $(\zeta,w)$, it follows from the estimates for the weighted Szeg\H{o} projection that $(P_{\varphi}\mathcal{A}_{\varphi}C_{\varphi,w})(z)$ is infinitely many times differentiable with respect to $w$ in $\overline{\Omega}$, and that
\[
\frac{\partial^k}{\partial w^k}(P_{\varphi}\mathcal{A}_{\varphi}C_{\varphi,w})(z)=\left(P_{\varphi}\left(\frac{\partial^k}{\partial w^k}\mathcal{A}_{\varphi}C_{\varphi,w}\right)\right)(z)
\]
for any non-negative integer $k$. Proceeding as before, we can show that for any non-negative integer $k$, the function $\frac{\partial^k}{\partial w^k}(P_{\varphi}\mathcal{A}_{\varphi}C_{\varphi,w})(z)$ and all its derivatives with respect to $z$ are continuous on $\overline{\Omega}\times\overline{\Omega}$.
Thus, we have proved that $H_2(z,w)=(P_{\varphi}\mathcal{A}_{\varphi}C_{\varphi,w})(z)\in C^{\infty}(\overline{\Omega}\times\overline{\Omega})$.
\end{proof}

\begin{thm}
Let $\Omega$ be a bounded $n$-connected domain with $C^{\infty}$ smooth boundary and $\varphi$ be a positive function in $C^{\infty}(\partial\Omega)$. Then, the function $l_{\varphi}(z,w)$ defined by 
\[
L_{\varphi}(z,w) = \frac{1}{2\pi(z-w)} + l_{\varphi}(z,w)
\]
is a function on $\Omega\times\Omega$ that is holomorphic in $z$ and $w$ and that extends to be in $C^{\infty}(\overline{\Omega}\times\overline{\Omega})$.
\end{thm}

\begin{proof}
By the properties of the weighted Garabedian kernel mentioned before, it is easy to see that $l_{\varphi}$ is holomorphic in $(z,w)\in\Omega\times\Omega$. 

\medskip

For a fixed $a\in\Omega$, the function $L_{\varphi}(\cdot,a)$ is holomorphic on $\Omega\setminus\{a\}$ with a simple pole at $z=a$ with residue $\frac{1}{2\pi}$, and extends $C^{\infty}$ smoothly to $\partial\Omega$. 
So, $l_{\varphi}(\cdot,a)\in A^{\infty}(\Omega)$. Define
\[
G_a(z) = \frac{1}{2\pi(z-a)}
\]
Recall that the functions in $L^2(\partial\Omega)$ orthogonal to $H^2_{1/\varphi}(\partial\Omega)$ are of the form $\varphi \overline{HT}$ where $H\in H^2(\partial\Omega)$. Therefore,
\begin{eqnarray*}
l_{\varphi}(\cdot,a) 
&=&
P_{1/\varphi}(l_{\varphi}(\cdot,a))
=
P_{1/\varphi}(L_{\varphi}(\cdot,a)) - P_{1/\varphi}G_a
\\
&=&
P_{1/\varphi}(i\,\varphi\, \overline{S_{\varphi}(\cdot,a)T}) - P_{1/\varphi}G_a
=
-P_{1/\varphi}G_a 
\\
&=&
P_{1/\varphi}\mathcal{A}_{1/\varphi} G_a - \mathcal{C} G_a.
\end{eqnarray*}
But for $z\in\Omega$,
\begin{eqnarray*}
(\mathcal{C}G_a)(z)
&=&
\frac{1}{2\pi i}\int_{\xi\in\partial\Omega} \frac{1}{2\pi}\frac{1}{(\xi - a)} \frac{1}{(\xi -z)} d\xi
\\
&=& \text{Residue} \left( \frac{1}{2\pi}\frac{1}{(\cdot - a)} \frac{1}{(\cdot -z)} ; \,a\right) + \text{Residue} \left( \frac{1}{2\pi}\frac{1}{(\cdot - a)} \frac{1}{(\cdot -z)} ;\, z\right)
\\
&=&
\frac{1}{2\pi(a-z)} + \frac{1}{2\pi(z-a)} = 0.
\end{eqnarray*}
Therefore,
\[
l_{\varphi}(z,a) = (P_{1/\varphi}\mathcal{A}_{1/\varphi} G_a) (z)
\]
Finally, proceeding as in the proof of Theorem \ref{smooth}, we obtain that $l_{\varphi}$ extends to be in $ C^{\infty}(\overline{\Omega}\times\overline{\Omega})$. 
\end{proof}

\section{Variation of $S_{\varphi}$ as a function of $\varphi$}

\noindent In this section, we will study the dependence of $S_{\varphi}$ on $\varphi$.

\begin{thm}\label{ram 1}
Let $\Omega\subset\mathbb{C}$ be a bounded $n$-connected domain with $C^{\infty}$ smooth boundary, and $\varphi$ be a positive real-valued $C^{\infty}$ function on $\partial\Omega$. Let $\{\varphi_k\}_{k=1}^{\infty}$ be a sequence of positive real-valued $C^{\infty}$ functions on $\partial\Omega$ such that $\varphi_k\rightarrow\varphi$ uniformly as $k\rightarrow\infty$ on $\partial\Omega$. Then
\[
\lim\limits_{k\rightarrow\infty} S_{\varphi_k}(z,w) = S_{\varphi}(z,w)
\]
locally uniformly on $\Omega\times\Omega$.
\end{thm}

\begin{proof}
Since $\varphi_k\rightarrow\varphi$ uniformly as $k\rightarrow\infty$ on $\partial\Omega$ and $\varphi$ is a positive real-valued $C^{\infty}$ function on $\partial\Omega$, there exists a constant $c_{\varphi}>0$ such that for large enough $k$, 
\[
{c^{-1}_{\varphi}} \leq \varphi_{k}(\zeta)\leq c_{\varphi} 
\]
for all $\zeta\in\partial\Omega$. Let us assume that $\partial\Omega$ has been parametrized with respect to arc length. Let $a\in\Omega$. For large enough $k$, we have
\[
\Vert C_a\varphi^{-1}_k\Vert_{\varphi_k}^2=\frac{1}{4\pi^2}\int_{\zeta\in\partial\Omega} \frac{1}{\vert \zeta-a\vert^2}\frac{1}{\varphi_k(\zeta)} ds \leq
c_{\varphi} \Vert C_a\Vert^2.
\]
The same inequality holds when $\varphi_k$ is replaced by $\varphi$. Therefore, for every $h\in H^2(\partial\Omega)$ and for large enough $k$, we have by (\ref{szego})
\[
\vert h(a)\vert \leq \sqrt{c_{\varphi}} \Vert C_a\Vert\, \Vert h\Vert_{\varphi_k}\quad\text{and}\quad\vert h(a)\vert \leq \sqrt{c_{\varphi}} \Vert C_a\Vert\, \Vert h\Vert_{\varphi}.
\]
That is, for large enough $k$, the evaluation linear functionals at $a$ on $H^2_{\varphi_k}(\partial\Omega)$ and $H^2_{\varphi}(\partial\Omega)$ are bounded with a uniform bound. This implies that
\[
\Vert S_{\varphi_k}(\cdot,a)\Vert_{\varphi_k} \leq \sqrt{c_{\varphi}} \Vert C_a\Vert \quad \text{and} \quad \Vert S_{\varphi}(\cdot,a)\Vert_{\varphi} \leq \sqrt{c_{\varphi}} \Vert C_a\Vert
\]
for large enough $k$. Therefore,
\begin{equation}\label{szego1}
\Vert S_{\varphi_k}(\cdot,a)\Vert^2=\int_{\zeta\in\partial\Omega}\vert S_{\varphi_k}(\zeta,a)\vert^2 ds \leq {c_{\varphi}} \int_{\zeta\in\partial\Omega}\vert S_{\varphi_k}(\zeta,a)\vert^2 \varphi_k(\zeta) ds \leq (c_{\varphi})^2\,\Vert C_a\Vert^2.
\end{equation}
The same inequality holds when $\varphi_k$ is replaced by $\varphi$. For large enough $k$, define the bounded linear functionals $\Lambda_k$ on $L^2(\partial\Omega)$ as
\[
f\mapsto \langle f, S_{\varphi_k}(\cdot,a)\rangle.
\]
It follows from (\ref{szego1}) that the linear functionals $\Lambda_k$ are bounded with a uniform bound. By the Banach-Alaoglu theorem, $\Lambda_k$ has a weak-${*}$ convergent subsequence. By passing to this subsequence, denote the weak-${*}$ limit by $\Lambda_0$, where 
\[
\Lambda_0: f\mapsto \langle f, S_0\rangle,\quad S_0\in L^2(\partial\Omega).
\]
Now, for $h\in H^2(\partial\Omega)$, we have
\[
h(a)=\int_{\zeta\in\partial\Omega} h(\zeta) \overline{S_{\varphi_k}(\zeta,a)}\,\varphi_k(\zeta)\,ds.
\]
Note that $h\varphi_k$ converges to $h\varphi$ in $L^2(\partial\Omega)$. Therefore,
\begin{multline*}
\left\lvert\langle h\varphi_k,S_{\varphi_k}(\cdot,a)\rangle-\langle h\varphi, S_0\rangle\right\rvert \\
= \left\lvert\langle h\varphi_k,S_{\varphi_k}(\cdot,a)\rangle-\langle h\varphi, S_{\varphi_k}(\cdot,a)\rangle+
\langle h\varphi, S_{\varphi_k}(\cdot,a)\rangle-\langle h\varphi, S_0\rangle\right\rvert\\
= \left\lvert\langle h\varphi_k-h\varphi, S_{\varphi_k}(\cdot,a)\rangle + \langle h\varphi,S_{\varphi_k}(\cdot,a)-S_0\rangle\right\rvert\\
\leq \left\lvert\langle h\varphi_k-h\varphi, S_{\varphi_k}(\cdot,a)\rangle\right\rvert + \left\lvert\langle h\varphi,S_{\varphi_k}(\cdot,a)-S_0\rangle\right\rvert\\
\leq \Vert h\varphi_k-h\varphi\Vert\,\Vert S_{\varphi_k}(\cdot,a)\Vert + \left\lvert\langle h\varphi,S_{\varphi_k}(\cdot,a)-S_0\rangle\right\rvert\\
\leq c_{\varphi}\,\Vert C_a\Vert\,\Vert h\varphi_k-h\varphi\Vert + \left\lvert\langle h\varphi,S_{\varphi_k}(\cdot,a)-S_0\rangle\right\rvert \rightarrow 0 \quad \text{as}\quad k\rightarrow\infty.
\end{multline*}
Thus,
\[
h(a)=\int_{\zeta\in\partial\Omega} h(\zeta) \overline{S_0(\zeta)}\,\varphi(\zeta)\,ds.
\]
If $f\in {H^2(\partial\Omega)}^{\perp}$, then $\langle f, S_{\varphi_k}(\cdot,a)\rangle=0$ for all $k$, and therefore $\langle f, S_0\rangle=0$. Thus, $S_0\in H^2(\partial\Omega)$ and hence, $S_0=S_{\varphi}(\cdot,a)$.

\smallskip

Now, for a compact $K\subset\Omega$, the Cauchy integral formula gives
\[
\sup_{\zeta\in K}\vert S_{\varphi_k}(\zeta,a)\vert \leq  C_K \Vert S_{\varphi_k}(\cdot,a)\Vert_{L^2(\partial\Omega)}.
\]
Since the $L^2(\partial\Omega)$-norm of the functions $S_{\varphi_k}(\cdot,a)$ is uniformly bounded, Montel's theorem says that the sequence $\{S_{\varphi_k}(\cdot,a)\}_{k=1}^{\infty}$ is a normal family of holomorphic functions on $\Omega$. Hence, it has a subsequence that converges to a holomorphic function $\Tilde{S}$, uniformly on all compact subsets of $\Omega$. Work with this subsequence in what follows. For all $h\in L^2(\partial\Omega)$
\[
\langle h, S_{\varphi_k}(\cdot,a)\rangle\rightarrow\langle h, S_{\varphi}(\cdot,a)\rangle\quad\text{as}\;k\rightarrow\infty.
\]
Recall that $S(\cdot,\cdot)$ denotes the Szeg\H{o} kernel of $\Omega$. Since $S(\cdot,\zeta)\in L^2(\partial\Omega)$ for every $\zeta\in\Omega$, we therefore have
\[
S_{\varphi_k}(\zeta,a)-S_{\varphi}(\zeta,a)=\langle S_{\varphi_k}(\cdot,a)-S_{\varphi}(\cdot,a),S(\cdot,\zeta)\rangle \rightarrow 0 \quad\text{as} \, k\rightarrow\infty.
\]
Hence, we must have that $\Tilde{S}=S_{\varphi}(\cdot,a)$. Thus, $S_{\varphi_k}(\cdot,a)$ converges locally uniformly to $S_{\varphi}(\cdot,a)$ on $\Omega$.

Moreover, if $K_1$ and $K_2$ are compact subsets of $\Omega$. Then, for $w\in K_2$ and large enough $k$
\[
\sup_{z\in K_1}\vert S_{\varphi_k}(z,w)\vert \leq C_{K_1} \Vert S_{\varphi_k}(\cdot,w) \Vert_{L^2(\partial\Omega)}\leq C_{K_1} c_{\varphi} \Vert C_w\Vert_{L^2(\partial\Omega)}.
\]
Since $K_2$ is compact, $\Vert C_w\Vert$ is uniformly bounded for $w\in K_2$. Therefore, there exists an $M>0$ such that
\[
\sup_{z\in K_1}\sup_{w\in K_2}\vert S_{\varphi_k}(z,w)\vert \leq M
\]
for large enough $k$. Hence, Montel's theorem says that $\{S_{\varphi_k}\}$ has a subsequence that converges locally uniformly to a function $H$ on $\Omega\times\Omega$ which is holomorphic in the first variable and antiholomorphic in the second variable. We must have $H=S_{\varphi}$ from the above discussion. So, $S_{\varphi_k}\rightarrow S_{\varphi}$ uniformly on all the compact subsets of $\Omega\times\Omega$.
\end{proof}

\begin{thm}\label{ram 2}
Let $\Omega\subset\mathbb{C}$ be a bounded $n$-connected domain with $C^{\infty}$ smooth boundary and $\varphi$ a positive real-valued $C^{\infty}$ smooth function on $\partial\Omega$. 
Let $\{\varphi_k\}_{k=1}^{\infty}$ be a sequence of positive real-valued $C^{\infty}$ functions on $\partial\Omega$ such that $\varphi_k\rightarrow\varphi$ in $C^{\infty}$ topology on $\partial\Omega$ as $k\rightarrow\infty$. Then 
\[
\lim_{k\rightarrow\infty}S_{\varphi_k}(z,w) = S_{\varphi}(z,w)
\]
locally uniformly on $(\Omega\times\overline{\Omega})\cup(\overline{\Omega}\times\Omega)$.
\end{thm}

\begin{proof}
Using the weighted Kerzman-Stein formula as before, we have for $z,w\in\Omega$,
\[
S_{\varphi_k}(z,w)
=
(\mathcal{C} C_{\varphi_k,w})(z) - (P_{\varphi_k} \mathcal{A}_{\varphi_k} C_{\varphi_k,w})(z)
=
G_{k}(z,w)-H_k(z,w)
\]
and similarly
\[
S_{\varphi}(z,w)
=
(\mathcal{C} C_{\varphi,w})(z) - (P_{\varphi} \mathcal{A}_{\varphi} C_{\varphi,w})(z)
=
G(z,w)-H(z,w).
\]
We first analyze $G_k(z,w)$ and prove that the sequence $G_k$ converges to $G$ locally uniformly on $(\overline{\Omega}\times\overline{\Omega})-\Delta$ where $\Delta=\{(z,z) : z\in\partial\Omega\}$ is the diagonal of the boundary.

\medskip

Let $D_r(z) \subset \mbb C$ denote the disc of radius $r > 0$ around $z \in \mbb C$. It suffices to show convergence on sets of the form $(D_r(z_0)\cap \overline{\Omega})\times (D_r(w_0)\cap\overline{\Omega})$ where $(z_0,w_0)\in(\overline{\Omega}\times\overline{\Omega})-\Delta$ and $r>0$ is small enough that $\overline{D_r(z_0)}\cap \overline{D_r(w_0)}=\emptyset$. Also, if $z_0\in\Omega$ (or $w_0\in\Omega$) then $r$ is chosen such that $\overline{D_r(z_0)}\cap \partial\Omega=\emptyset$ (or $\overline{D_r(w_0)}\cap \partial\Omega=\emptyset$).

\medskip

Let $\chi$ be a function in $C^{\infty}(\partial\Omega)$ such that $\chi\equiv 1$ on a neighborhood of $\overline{D_r(z_0)}\cap \partial\Omega$ and $\chi\equiv 0$ on a neighborhood of $\overline{D_r(w_0)}\cap \partial\Omega$. In case $w_0\in\Omega$, take $\chi\equiv 1$. If $w_0\in\partial\Omega$ and $z_0\in\Omega$, take $\chi\equiv 0$. For $z,w\in\Omega$,
\[
G_k(z,w)-G(z,w)=(\mathcal{C}(\chi C_{\varphi_k,w} - \chi C_{\varphi,w}))\,(z) + (\mathcal{C}((1-\chi)C_{\varphi_k,w} - (1-\chi)C_{\varphi,w}))\,(z).
\]
There exists a positive integer $n$ and a constant $K>0$ such that $\Vert \mathcal{C}u\Vert_0\leq K \Vert u\Vert_{n}$ for all $u\in C^{\infty}(\partial\Omega)$.
For $\epsilon>0$, there exists a positive integer $k_1$ such that $\Vert \chi(C_{\varphi_k,w} - C_{\varphi,w})\Vert_n<\epsilon$ for all  $k\geq k_1$ and $w\in D_r(w_0)\cap \overline{\Omega}$. Thus, for all $z\in \overline{\Omega}$, $w\in D_r(w_0)\cap \overline{\Omega}$ and $k\geq k_1$, we have
\begin{eqnarray*}
\vert \mathcal{C}(\chi C_{\varphi_k,w} - \chi C_{\varphi,w})\vert(z) 
&\leq&
\sup_{z\in\partial\Omega}\vert \mathcal{C}(\chi C_{\varphi_k,w} - \chi C_{\varphi,w})\vert(z)\\
&=&
\Vert \mathcal{C}(\chi C_{\varphi_k,w} - \chi C_{\varphi,w})\Vert_0
\leq
K \Vert \chi(C_{\varphi_k,w} - C_{\varphi,w})\Vert_n<K\epsilon.
\end{eqnarray*}
Note that for $z,w\in\Omega$, we have 
\[
(\mathcal{C}((1-\chi)C_{\varphi_k,w} - (1-\chi)C_{\varphi,w}))\,(z)
=
\overline{(\mathcal{C}((1-\overline{\chi})C_{\varphi_k,z} - (1-\overline{\chi})C_{\varphi,z}))\,(w)}.
\]
There exists a positive integer $k_2$ such that $\Vert (1-\overline{\chi})C_{\varphi_k,z} - (1-\overline{\chi})C_{\varphi,z}\Vert_n<\epsilon$ for all $k\geq k_2$ and $z\in D_r(z_0)\cap\overline{\Omega}$. Thus, for all $z\in D_r(z_0)\cap\overline{\Omega}$, $w\in \overline{\Omega}$ and $k\geq k_2$, we have
\begin{eqnarray*}
\vert \mathcal{C}((1-\chi) C_{\varphi_k,w} - (1-\chi) C_{\varphi,w})\vert(z) &=&
\vert \mathcal{C}((1-\overline{\chi}) C_{\varphi_k,z} - (1-\overline{\chi}) C_{\varphi,z})\vert(w)\\
&\leq&
\sup_{w\in\partial\Omega}\vert \mathcal{C}((1-\overline{\chi}) C_{\varphi_k,z} - (1-\overline{\chi}) C_{\varphi,z})\vert(w)\\
&=&
\Vert \mathcal{C}((1-\overline{\chi}) C_{\varphi_k,z} - (1-\overline{\chi}) C_{\varphi,z})\Vert_0\\
&\leq&
K \Vert (1-\overline{\chi})C_{\varphi_k,z} - (1-\overline{\chi})C_{\varphi,z})\Vert_n<K\epsilon.
\end{eqnarray*}
Hence, for all $z\in D_r(z_0)\cap\overline{\Omega}$, $w\in D_r(w_0)\cap\overline{\Omega}$ and $k\geq \max\{k_1,k_2\}$, we have
\[
\vert G_k(z,w)-G(z,w)\vert <2K\epsilon.
\]
Thus, $G_k(z,w)\rightarrow G(z,w)$ as $k\rightarrow\infty$ uniformly on $(D_r(z_0)\cap\overline{\Omega})\times (D_r(w_0)\cap\overline{\Omega})$. So, we have shown that $G_k(z,w)\rightarrow G(z,w)$ as $k\rightarrow\infty$ locally uniformly on $(\overline{\Omega}\times\overline{\Omega})-\Delta$.

\medskip

We shall now analyze $H_k(z,w)=(P_{\varphi_k} \mathcal{A}_{\varphi_k} C_{\varphi_k,w})(z)$ and prove that the sequence $H_k$ converges to $H$ locally uniformly on $(\Omega\times\overline{\Omega})$. Let $\psi_{\zeta}^k(\xi)=\overline{A_{\varphi_k}(\zeta,\xi)}$ and $\psi_{\zeta}(\xi)=\overline{A_{\varphi}(\zeta,\xi)}$ where $\zeta,\xi\in\partial\Omega$. For $w\in\overline{\Omega}$ and $\zeta\in\partial\Omega$,
\begin{eqnarray*}
\vert(\mathcal{A}_{\varphi_k}C_{\varphi_k,w})(\zeta)-(\mathcal{A}_{\varphi}C_{\varphi,w})(\zeta)\vert 
&=&
\vert\overline{(\mathcal{C}\psi^k_{\zeta})(w)}-\overline{(\mathcal{C}\psi_{\zeta})(w)}\vert
=
\vert\mathcal{C}(\psi^k_{\zeta}-\psi_{\zeta})\vert(w)
\leq
\Vert\mathcal{C}(\psi^k_{\zeta}-\psi_{\zeta})\Vert_0\\
&\leq&
K\Vert \psi^k_{\zeta}-\psi_{\zeta}\Vert_n.
\end{eqnarray*}
It can be easily checked that $\Vert \psi^k_{\zeta}-\psi_{\zeta}\Vert_n\rightarrow 0$ as $k\rightarrow\infty$. Thus, $(\mathcal{A}_{\varphi_k}C_{\varphi_k,w})(\zeta)\rightarrow (\mathcal{A}_{\varphi}C_{\varphi,w})(\zeta)$ uniformly on $\partial\Omega\times \overline{\Omega}$ as $k\rightarrow\infty$ when considered as a function of $(\zeta,w)$. Similarly, for a non-negative integer $m$, it can be shown that
\[
\frac{\partial^m}{\partial t^m} (\mathcal{A}_{\varphi_k}C_{\varphi_k,w})(\zeta(t))
\rightarrow
\frac{\partial^m}{\partial t^m}(\mathcal{A}_{\varphi}C_{\varphi,w})(\zeta(t))
\]
uniformly on $\partial\Omega\times \overline{\Omega}$ as $k\rightarrow\infty$ when considered as a function of $(\zeta,w)$. Now, for $z,w\in\overline{\Omega}$
\begin{multline*}
\vert (P_{\varphi_k}\mathcal{A}_{\varphi_k} C_{\varphi_k,w})(z)-(P_{\varphi}\mathcal{A}_{\varphi} C_{\varphi,w})(z)\vert
\\
=
\vert (P_{\varphi_k}\mathcal{A}_{\varphi_k} C_{\varphi_k,w})(z)-(P_{\varphi_k}\mathcal{A}_{\varphi} C_{\varphi,w})(z)\vert 
+
\vert (P_{\varphi_k}\mathcal{A}_{\varphi} C_{\varphi,w})(z)-(P_{\varphi}\mathcal{A}_{\varphi} C_{\varphi,w})(z)\vert\\
=
\vert (P_{\varphi_k}(\mathcal{A}_{\varphi_k} C_{\varphi_k,w}-\mathcal{A}_{\varphi} C_{\varphi,w}))(z)\vert 
+
\vert ((P_{\varphi_k}-P_{\varphi})\mathcal{A}_{\varphi} C_{\varphi,w})(z)\vert.
\end{multline*}
Observe that
\begin{eqnarray*}
\vert (P_{\varphi_k}(\mathcal{A}_{\varphi_k} C_{\varphi_k,w}-\mathcal{A}_{\varphi} C_{\varphi,w}))(z)\vert 
&\leq& 
\Vert P_{\varphi_k}(\mathcal{A}_{\varphi_k} C_{\varphi_k,w}-\mathcal{A}_{\varphi} C_{\varphi,w}) \Vert_0\\
&\leq& 
(K+C_k) \Vert \mathcal{A}_{\varphi_k} C_{\varphi_k,w}-\mathcal{A}_{\varphi} C_{\varphi,w}\Vert_n,
\end{eqnarray*}
where
\[
C_k=\sup\left\lbrace\left\lvert \frac{d^m}{dt^m} A_{\varphi_k}(z(t),\zeta) \right\rvert:t\in J,\,0\leq m\leq 1,\,\zeta\in\partial\Omega\right\rbrace 
\int_{\zeta\in\partial\Omega}\varphi_k(\zeta)\,ds.
\]
Since $\varphi_k\rightarrow\varphi$ uniformly on $\partial\Omega$ and $A_{\varphi_k}\rightarrow A_{\varphi}$ uniformly on $\partial\Omega\times\partial\Omega$, the constants $C_k$ are bounded. We have seen above that $\Vert \mathcal{A}_{\varphi_k} C_{\varphi_k,w}-\mathcal{A}_{\varphi} C_{\varphi,w}\Vert_n\rightarrow 0$ as $k\rightarrow\infty$ uniformly for all $w\in\overline{\Omega}$. Therefore,
\[
(P_{\varphi_k}(\mathcal{A}_{\varphi_k} C_{\varphi_k,w}-\mathcal{A}_{\varphi} C_{\varphi,w}))(z)\rightarrow 0
\quad \text{as}\,k\rightarrow\infty
\]
uniformly on $\overline{\Omega}\times\overline{\Omega}$. 

\medskip

Let $\mathcal{K}$ be a compact subset of $\Omega$. For $z\in\mathcal{K}$ and $w\in\overline{\Omega}$
\begin{multline*}
\vert ((P_{\varphi_k}-P_{\varphi})\mathcal{A}_{\varphi} C_{\varphi,w})(z)\vert\\
=
\left\lvert
\int_{\zeta\in\partial\Omega} S_{\varphi_k}(z,\zeta)\,\mathcal{A}_{\varphi} C_{\varphi,w}(\zeta)\,\varphi_k(\zeta) \,ds
-
\int_{\zeta\in\partial\Omega} S_{\varphi}(z,\zeta)\,\mathcal{A}_{\varphi} C_{\varphi,w}(\zeta)\varphi(\zeta) \,ds 
\right\rvert
\\
\leq
\left\lvert
\int_{\partial\Omega} (S_{\varphi_k}(z,\zeta)-S_{\varphi}(z,\zeta))\mathcal{A}_{\varphi} C_{\varphi,w}(\zeta) \varphi_k(\zeta)\,ds
\right\rvert 
+
\left\lvert
\int_{\partial\Omega} S_{\varphi}(z,\zeta)\mathcal{A}_{\varphi} C_{\varphi,w}(\zeta) (\varphi_k-\varphi)(\zeta)\,ds
\right\rvert.
\end{multline*}
Since $S_{\varphi}(z,\zeta)\in C^{\infty}(\mathcal{K}\times \partial\Omega)$, $\mathcal{A}_{\varphi} C_{\varphi,w}(\zeta)\in C^{\infty}(\partial\Omega\times \overline{\Omega})$ as a function of $(\zeta,w)$ and $\varphi_k\rightarrow\varphi$ uniformly on $\partial\Omega$, 
\[
\left\lvert
\int_{\zeta\in\partial\Omega} S_{\varphi}(z,\zeta)\mathcal{A}_{\varphi} C_{\varphi,w}(\zeta) (\varphi_k-\varphi)(\zeta)\,ds
\right\rvert
\rightarrow 0
\quad \text{as}\,k\rightarrow\infty
\]
uniformly on $\mathcal{K}\times\overline{\Omega}$. Let 
\[
R_k(z,\zeta)
=
S_{\varphi_k}(z,\zeta)-S_{\varphi}(z,\zeta), 
\quad
z\in\mathcal{K},\, \zeta\in\partial\Omega.
\]
Observe that
\begin{eqnarray*}
\left\lvert
\int_{\partial\Omega} R_k(z,\zeta)\mathcal{A}_{\varphi} C_{\varphi,w}(\zeta) \varphi_k(\zeta)\,ds
\right\rvert
&\leq&
\sqrt{\int_{\partial\Omega} \vert R_k(z,\zeta)\vert^2 \varphi_k(\zeta)\,ds}
\sqrt{\int_{\partial\Omega} \vert \mathcal{A}_{\varphi} C_{\varphi,w}(\zeta)\vert^2 \varphi_k(\zeta)\,ds}.
\end{eqnarray*}
Since $\varphi_k\rightarrow\varphi$ uniformly on $\partial\Omega$ and $\mathcal{A}_{\varphi} C_{\varphi,w}(\zeta)\in C^{\infty}(\partial\Omega\times\overline{\Omega})$ as a function of $(\zeta,w)$, the function
\[
\int_{\partial\Omega} \vert \mathcal{A}_{\varphi} C_{\varphi,w}(\zeta)\vert^2 \varphi_k(\zeta)\,ds
\]
is uniformly bounded for $w\in\overline{\Omega}$. Now,
\begin{eqnarray*}
\int_{\zeta\in\partial\Omega} \vert R_k(z,\zeta)\vert^2 \varphi_k(\zeta)\,ds
&=&
\int_{\partial\Omega} (S_{\varphi_k}(z,\zeta)-S_{\varphi}(z,\zeta)) R_k(\zeta,z) \varphi_k(\zeta)\,ds\\
&=&
\int_{\partial\Omega} S_{\varphi_k}(z,\zeta) R_k(\zeta,z) \varphi_k(\zeta)\,ds
-
\int_{\partial\Omega} S_{\varphi}(z,\zeta) R_k(\zeta,z) \varphi_k(\zeta)\,ds\\
&=&
R_k(z,z)
-
\int_{\partial\Omega} S_{\varphi}(z,\zeta) R_k(\zeta,z) (\varphi_k-\varphi+\varphi)(\zeta)\,ds\\
&=&
\int_{\partial\Omega} S_{\varphi}(z,\zeta) R_k(\zeta,z) (\varphi_k-\varphi)(\zeta)\,ds\\
&\leq&
\sqrt{\int_{\partial\Omega} \vert S_{\varphi}(z,\zeta)\vert^2 \vert\varphi_k-\varphi\vert(\zeta)\,ds}
\sqrt{\int_{\partial\Omega} \vert R_k(\zeta,z)\vert^2 \vert\varphi_k-\varphi\vert(\zeta)\,ds}\\
&\leq&
\sqrt{M_k}
\sqrt{\int_{\partial\Omega} \vert S_{\varphi}(z,\zeta)\vert^2 \varphi(\zeta)\,ds}
\sqrt{\int_{\partial\Omega} \vert R_k(\zeta,z)\vert^2 \varphi_k(\zeta)\,ds}\\
&=&
\sqrt{M_k} \sqrt{S_{\varphi}(z,z)}
\sqrt{\int_{\partial\Omega} \vert R_k(\zeta,z)\vert^2 \varphi_k(\zeta)\,ds},
\end{eqnarray*}
where
\[
M_k=\max\left\lbrace \frac{\vert \varphi_k-\varphi\vert^2(\zeta)}{\varphi_k(\zeta)\,\varphi(\zeta)}: \zeta\in\partial\Omega\right\rbrace.
\]
Therefore,
\[
\int_{\zeta\in\partial\Omega} \vert R_k(z,\zeta)\vert^2 \varphi_k(\zeta)\,ds
\leq 
M_k\,
S_{\varphi}(z,z),\quad z\in\mathcal{K}.
\]
Since $\varphi_k$ converges to $\varphi$ uniformly on $\partial\Omega$, it follows that $M_k\rightarrow 0$ as $k\rightarrow\infty$. Since $S_{\varphi}(z,z)$ is bounded for $z\in \mathcal{K}$, 
\[
\int_{\zeta\in\partial\Omega} \vert R_k(z,\zeta)\vert^2 \varphi_k(\zeta)\,ds
\rightarrow 0 \quad \text{as}\,k\rightarrow\infty
\]
uniformly for $z\in\mathcal{K}$. Thus,
\[
\left\lvert
\int_{\partial\Omega} R_k(z,\zeta)\mathcal{A}_{\varphi} C_{\varphi,w}(\zeta) \varphi_k(\zeta)\,ds
\right\rvert
\rightarrow 0 
\quad \text{as}\,
k\rightarrow \infty
\]
uniformly on $\mathcal{K}\times\overline{\Omega}$. Since $\mathcal{K}$ is an arbitrary compact subset of $\Omega$, we see that
\[
((P_{\varphi_k}-P_{\varphi})\mathcal{A}_{\varphi} C_{\varphi,w})(z)\rightarrow 0
\quad
\text{as}\, k\rightarrow\infty
\]
locally uniformly on $\Omega\times\overline{\Omega}$. Thus, we have proved that $H_k(z,w)\rightarrow H(z,w)$ as $k\rightarrow\infty$ locally uniformly on $\Omega\times\overline{\Omega}$. Combining the convergence results for $G_k$ and $H_k$, and noting that the Szeg\H{o} kernel is conjugate symmetric, we finally conclude that
\[
\lim_{k\rightarrow\infty} S_{\varphi_k}(z,w)=S_{\varphi}(z,w)
\]
locally uniformly on $(\Omega\times\overline{\Omega})\cup (\overline{\Omega}\times\Omega)$.
\end{proof}

\begin{cor}\label{ram2_cor}
Assume that the hypotheses of Theorem \ref{ram 2} hold. Let $f:\Omega\rightarrow\mathbb{D}$ be a proper holomorphic map having simple zeroes. Then
\[
\lim_{k\rightarrow\infty}S_{\varphi_k}(z,w) = S_{\varphi}(z,w)
\]
locally uniformly on $(\overline{\Omega}\times\overline{\Omega})- \mathcal{L}$ where $\mathcal{L}=\{(z,w)\in\partial\Omega\times\partial\Omega : f(z) \overline{f(w)}=1\}$. 
\end{cor}

\begin{proof}
Let $\{a_1,a_2,\ldots,a_N\}$ be the zero set of $f$. By \cite{Be1}, it is known that 
\begin{equation}
S_{\varphi_k}(z,w) = \frac{1}{1-f(z)\overline{f(w)}} \sum_{i,j=1}^N c_{ijk}\, S_{\varphi_k}(z,a_i) \overline{S_{\varphi_k}(w,a_j)} 
\end{equation}
for $z,w\in\Omega$, where the coefficients $c_{ijk}$ are determined by the condition $[c_{ijk}]=[S_{\varphi_k}(a_i,a_j)]^{-1}$, and
\[
S_{\varphi}(z,w) = \frac{1}{1-f(z)\overline{f(w)}} \sum_{i,j=1}^N c_{ij}\, S_{\varphi}(z,a_i) \overline{S_{\varphi}(w,a_j)} 
\]
where $[c_{ij}]=[S_{\varphi}(a_i,a_j)]^{-1}$. Note that $\Delta\subset\mathcal{L}$. The above relations hold for all $(z,w)\in (\overline{\Omega}\times\overline{\Omega})-\mathcal{L}$ by continuity. 
\end{proof}

\begin{rem}

First, for $p \in \Omega$, the Ahlfors map $f_p$, which is the solution to the extremal problem 
    \[
    sup\{f'(p) \, : \, f:\Omega\rightarrow\mathbb{D} \text{ is holomorphic and }f'(p)>0\},
    \]
    is a proper holomorphic map. Furthermore, $f_p$ has simple zeroes for $p$ close to the boundary $\partial\Omega$ -- see \cite{Be0}, and is a candidate for the map $f$ that is used in the corollary above.    

\medskip
    
    Second, the proof of Theorem \ref{ram 2} provides a different way to prove the local uniform convergence of the Szeg\H{o} kernels $S_{\varphi_k}$ to the Szeg\H{o} kernel $S_{\varphi}$ in $\Omega\times\Omega$, but under a much stronger condition, namely, $\varphi_k\rightarrow\varphi$ in the $C^{\infty}$ topology on $\partial\Omega$ as $k\rightarrow\infty$.

\end{rem}

\begin{thm}\label{ram 3}
Let $\Omega\subset\mathbb{C}$ be a bounded $n$-connected domain with $C^{\infty}$ smooth boundary and $\varphi$ a positive real-valued $C^{\infty}$ smooth function on $\partial\Omega$. 
Let $\{\varphi_k\}_{k=1}^{\infty}$ be a sequence of positive real-valued $C^{\infty}$ functions on $\partial\Omega$ such that $\varphi_k\rightarrow\varphi$ in the $C^{\infty}$ topology on $\partial\Omega$ as $k\rightarrow\infty$. Let $f:\Omega\rightarrow\mathbb{D}$ be a proper holomorphic map with simple zeroes. Then 
\[
\lim_{k\rightarrow\infty}l_{\varphi_k}(z,w) = l_{\varphi}(z,w)
\]
locally uniformly on $(\overline{\Omega}\times\overline{\Omega})- \mathcal{L}$ where $\mathcal{L}=\{(z,w)\in\partial\Omega\times\partial\Omega : f(z) \overline{f(w)}=1\}$. In particular, we have local uniform convergence on $(\Omega\times\overline{\Omega})\cup(\overline{\Omega}\times\Omega)$.
\end{thm}

\begin{proof}
Let $K\subset\Omega$ be compact. Note that for $w\in K$, $L_{\varphi_k}(\cdot,w)-L_{\varphi}(\cdot,w)$ is a holomorphic function on $\Omega$ as the poles cancel out, and it extends $C^{\infty}$ smoothly to $\partial\Omega$. By the maximum modulus principle,
\begin{eqnarray*}
\sup\limits_{\substack{z\in\overline{\Omega}\\ w\in K}}\vert L_{\varphi_k}(z,w)-L_{\varphi}(z,w)\vert
&=&    
\sup\limits_{\substack{z\in\partial\Omega\\ w\in K}}\vert L_{\varphi_k}(z,w)-L_{\varphi}(z,w)\vert\\
&=&
\sup\limits_{\substack{z\in\partial\Omega\\ w\in K}}\vert i\varphi_k(z)S_{\varphi_k}(w,z)\overline{T(z)}-i\varphi(z)S_{\varphi_k}(w,z)\overline{T(z)}\vert\\
&=&
\sup\limits_{\substack{z\in\partial\Omega\\ w\in K}}\vert \varphi_k(z)S_{\varphi_k}(w,z)-\varphi(z)S_{\varphi_k}(w,z)\vert\\
& \leq &
\sup_{\substack{z\in \partial\Omega\\ w\in K}}\left(\varphi_k(z)\vert S_{\varphi_k}(w,z) - S_{\varphi}(w,z) \vert + \vert S_{\varphi}(w,z) (\varphi_k(z) - \varphi(z)) \vert\right)
\end{eqnarray*}
Since $\varphi_k\rightarrow \varphi$ uniformly on $\partial\Omega$ and $S_{\varphi}\in C^{\infty}(K\times \overline{\Omega})$, there exists a constant $M>0$ and $k_0 \ge 1$ such that
\[
\sup_{\substack{z\in \overline{\Omega}\\ w\in K}}\vert L_{\varphi_k}(z,w) - L(z,w)\vert 
\leq 
M \sup_{\substack{z\in \partial\Omega\\ w\in K}}\left(\vert S_{\varphi_k}(w,z) - S_{\varphi}(w,z) \vert + \vert \varphi_k(z) - \varphi(z) \vert\right)
\]
for all $k\geq k_0$. By Theorem \ref{ram 2} and (\ref{garabedian}), it follows that
\begin{equation}\label{gar}
\lim_{k\rightarrow\infty} l_{\varphi_k}(z,w) = l_{\varphi}(z,w)
\end{equation}
locally uniformly on $(\Omega\times\overline{\Omega}) \cap (\overline{\Omega}\times \Omega)$.

\medskip

Let $\{a_1,a_2,\ldots,a_N\}$ be the zero set of $f$. Recall that for $(z,w)\in(\overline{\Omega}\times\overline{\Omega})- \mathcal{L}$,
\[
S_{\varphi_k}(z,w) = \frac{1}{1-f(z)\overline{f(w)}} \sum_{i,j=1}^N c_{ijk}\, S_{\varphi_k}(z,a_i) \overline{S_{\varphi_k}(w,a_j)} 
\]
where the coefficients $c_{ijk}$ are determined by the condition $[c_{ijk}]=[S_{\varphi_k}(a_i,a_j)]^{-1}$. For $(z,w)\in (\Omega \times \partial\Omega)$, this can be rewritten as
\[
\frac{1}{i\varphi_k(w)}L_{\varphi_k}(w,z)T(w) = \frac{1}{1-f(z)\overline{f(w)}} \sum_{i,j=1}^N c_{ijk}\, S_{\varphi_k}(z,a_i) \frac{1}{i\varphi_k(w)} L_{\varphi_k}(w,a_j) T(w).
\]
That is,
\begin{equation}
L_{\varphi_k}(w,z) = \frac{f(w)}{f(w)-f(z)} \sum_{i,j=1}^N c_{ijk}\, S_{\varphi_k}(z,a_i) L_{\varphi_k}(w,a_j), 
\end{equation}
where $[c_{ijk}]=[S_{\varphi_k}(a_i,a_j)]^{-1}$. Similarly,
\[
L_{\varphi}(w,z) = \frac{f(w)}{f(w)-f(z)} \sum_{i,j=1}^N c_{ij}\, S_{\varphi}(z,a_i) L_{\varphi}(w,a_j),
\]
where $[c_{ij}]=[S_{\varphi}(a_i,a_j)]^{-1}$. Therefore,
\begin{equation}\label{gar1}
L_{\varphi_k}(w,z) - L_{\varphi}(w,z) = \frac{f(w)}{f(w)-f(z)}
\sum_{i,j=1}^N 
c_{ijk}\, S_{\varphi_k}(z,a_i) L_{\varphi_k}(w,a_j) 
- 
c_{ij}\, S_{\varphi}(z,a_i) L_{\varphi}(w,a_j)
\end{equation}
Note that $L_{\varphi_k}(w,z) - L_{\varphi}(w,z)$ is a holomorphic function in $z$ and $w$ and extends to be in $C^{\infty}(\overline{\Omega}\times \overline{\Omega})$. For a fixed $z\in\Omega$, the function on the right side in (\ref{gar1}) is holomorphic on $\Omega - \{z,a_1,\ldots,a_N\}$ and extends $C^{\infty}$ smoothly to $\partial\Omega$. Therefore, $z,a_1,\ldots,a_N$ must be removable singularities. Since (\ref{gar1}) holds for all $w\in\partial\Omega$, the identity principle implies that it also holds for all $w\in \Omega - \{z,a_1,\ldots,a_N\}$. A final continuity argument gives that (\ref{gar1}) is true for all $(z,w)\in (\overline{\Omega}\times (\overline{\Omega}-\{a_1,\ldots,a_N\}))-\mathcal{L}$ with $z\neq w$. Hence the theorem follows using (\ref{gar}), (\ref{gar1}) and Theorem \ref{ram 2}.
\end{proof}

\section{Weights close to the constant weight $\mathbf{1}$}\label{weights_constant}

Since several aspects of $S(z, a)$ (which corresponds to $\varphi \equiv 1$) are known, it is reasonable to expect a better understanding of the map $\varphi \mapsto S_{\varphi}$ at $\varphi \equiv 1$. Theorem \ref{main} in this section validates this belief to some extent. We begin with:

\begin{lem}\label{density}
Let $\Omega$ be a bounded $n$-connected domain with $C^{\infty}$ smooth boundary, $\varphi$ a positive $C^{\infty}$ smooth function on $\partial\Omega$. Then
\[
\Sigma_{\varphi} = {\text Span}_{\mbb C}\left\lbrace S_{\varphi}(\cdot,a) : a\in\Omega \right\rbrace
\]
is dense in $A^{\infty}(\Omega)$. 
\end{lem}

In the unweighted case, this is exactly Theorem $9.1$ in \cite{Be0}. With minor changes, the same proof works in the weighted case as well. The details are omitted.

\medskip

For a domain $\Omega\subset\mathbb{C}$, let $\Hat{\Omega}$ denote the double of $\Omega$ and $R(z)$ denote the antiholomorphic involution on $\Hat{\Omega}$ which fixes boundary $\partial\Omega$. Let $\Tilde{\Omega}=R(\Omega)$ denote the reflection of $\Omega$ in $\Hat{\Omega}$ across the boundary.

\medskip

It is known that (see \cite{Be0}) that if $g$ and $h$ are meromorphic functions on $\Omega$ which extend continuously to the boundary such that $g(z) = \overline{h(z)}$ for $z\in \partial\Omega$, then $g$ extends meromorphically to the double $\Hat{\Omega}$ with the extension $\hat{g}$ given by
\[
\hat{g}(z) = 
\begin{cases}
g(z) & z\in \Omega \sqcup \partial\Omega\\
h(R(z)) & z\in \Tilde{\Omega}
\end{cases}
\]
Any proper holomorphic map $f:\Omega\rightarrow\mathbb{D}$ extends smoothly to $\partial\Omega$. Since $f(z) = 1/\overline{f(z)}$ for $z\in\partial\Omega$, we see that $f$ extends meromorphically to the double $\Hat{\Omega}$. 

\medskip

\begin{thm}\label{Bell1}
Let $\Omega$ be a bounded $n$-connected domain with $C^{\infty}$ smooth boundary and $\varphi$ a positive $C^{\infty}$ smooth function on $\partial\Omega$. For any $a\in\Omega$, the zeroes of $S_{\varphi}(\cdot, a)$ on the boundary must be isolated and of finite order.
\end{thm}

\begin{proof}
Consider the constant function $\mathbf{1}\in A^{\infty}(\Omega)$. By Lemma \ref{density}, there exist constants $c_i\in \mbb C$ and points $a_i\in\Omega$, $1 \le i \le k$ such that
\[
\sup_{z\in\overline{\Om}} \left\lvert 1 - \sum_{i=1}^k c_i S_{\varphi}(z,a_i)\right\rvert < \frac{1}{2}.
\]
Thus, $\Sigma = \sum_{i=1}^k c_i S_{\varphi}(\cdot,a_i)$ is non-vanishing on $\overline{\Omega}$. Let $a\in\Omega$ be arbitrary. The function 
\[
f = \Sigma^{-1} \; S_{\varphi}(\cdot,a)
\]
is smooth on $\overline{\Omega}$. Recall that for $w\in\Omega$
\[
\varphi(z)\, S_{\varphi}(z,w) = i \,\overline{L_{\varphi}(z,w) T(z)}, \quad z\in\partial\Omega.
\]
Therefore, for $z\in\partial\Omega$, we have
\begin{eqnarray*}
f(z) 
&=&
\frac{S_{\varphi}(z,a)}{\sum_{i=1}^k c_i S_{\varphi}(z,a_i)} 
=
\frac{\varphi(z) \, S_{\varphi}(z,a)}{\sum_{i=1}^k c_i\, \varphi(z)\,S_{\varphi}(z,a_i)}
\\ &=&
\frac{i \,\overline{L_{\varphi}(z,a) T(z)}}{\sum_{i=1}^k c_i\, i \,\overline{L_{\varphi}(z,a_i) T(z)}}
=
\frac{\overline{L_{\varphi}(z,a)}}{\sum_{i=1}^k c_i\, \overline{L_{\varphi}(z,a_i)}}\\
&=&
\overline{\left(\frac{L_{\varphi}(R(z),a)}{\sum_{i=1}^k \overline{ c_i}\,L_{\varphi}(R(z),a_i)}\right)}.
\end{eqnarray*}
Note that
\[
g = \overline{\left(\frac{L_{\varphi}(R(\cdot),a)}{\sum_{i=1}^k \overline{ c_i}\,L_{\varphi}(R(\cdot),a_i)}\right)}
\]
is a meromorphic function on $\Tilde{\Omega}=R(\Omega)$ that extends smoothly to the boundary. Also, $f=g$ on $\partial\Omega$. Therefore, $f$ extends meromorphically to the double $\Hat{\Om}$, and hence the zeroes of $S_{\varphi}(\cdot,a)$ on $\partial\Omega$ must be isolated and of finite order.
\end{proof}

\begin{cor}\label{Bell2}
Let $\Omega$ be a bounded $n$-connected domain with $C^{\infty}$ smooth boundary and $\varphi$ a positive real-valued $C^{\infty}$ smooth function on $\partial\Omega$. For any two points $A_1$ and $A_2$ in $\Omega$, the function $S_{\varphi}(z,A_0)/S_{\varphi}(z,A_1)$ extends meromorphically to the double of $\Omega$.
\end{cor}
\begin{proof}
It can be checked that for $z\in\partial\Om$,
\[
\frac{S_{\varphi}(z,A_0)}{S_{\varphi}(z,A_1)}
=
\overline{\left(\frac{L_{\varphi}(z,A_0)}{L_{\varphi}(z,A_1)}\right)} 
=
\overline{\left(\frac{L_{\varphi}(R(z),A_0)}{L_{\varphi}(R(z),A_1)}\right)}.
\]
Also, the function 
\[
\overline{\left(\frac{L_{\varphi}(R(z),A_0)}{L_{\varphi}(R(z),A_1)}\right)}
\]
is meromorphic on $\Tilde{\Om}=R(\Omega)$. Since the zeroes of functions of $z$ of the form $S_{\varphi}(z,a)$ on $\partial\Om$ are isolated and of finite order, the function $S_{\varphi}(z,A_0)/S_{\varphi}(z,A_1)$ extends to the boundary with at most finitely many pole like singularities. Thus, $S_{\varphi}(z,A_0)/S_{\varphi}(z,A_1)$ extends meromorphically to the double of $\Omega$.
\end{proof}

\begin{thm}\label{main}
Let $\Omega\subset\mathbb{C}$ be a bounded $n$-connected domain with $C^{\infty}$ smooth boundary and $\varphi_k$ be a sequence of positive $C^{\infty}$ smooth functions on $\partial\Omega$ such that $\varphi_k\rightarrow \mathbf{1}$ in $C^{\infty}$ topology on $\partial\Omega$ as $k\rightarrow\infty$. Then for every $w_0\in\partial\Omega$
\[
\lim_{k\rightarrow\infty}S_{\varphi_k}(z,w_0) = S(z,w_0)
\]
locally uniformly on $\overline{\Omega}\setminus\{w_0\}$. 
\end{thm}

\begin{proof}
It is enough to show the convergence near points on $\partial\Omega$ other than $w_0$. Let $f:\Omega\rightarrow\mathbb{D}$ be a proper holomorphic map with simple zeroes. By Theorem \ref{ram 2} and Corollary $\ref{ram2_cor}$,
\begin{equation}\label{con1}
\lim_{k\rightarrow\infty}S_{\varphi_k}(z,w) = S(z,w)
\end{equation}
locally uniformly on $(\overline{\Omega}\times\overline{\Omega})- \mathcal{L}$ where $\mathcal{L}=\{(z,w)\in\partial\Omega\times\partial\Omega : f(z) \overline{f(w)}=1\}$. Let $b\in\Omega$ be fixed. Since $S(z,b)$ does not vanish for $z\in\partial\Omega$, there exists a $k_1 \ge 1$ such that $S_{\varphi_k}(z,b)$ does not vanish for $z\in\partial\Omega$ and $k\geq k_1$.

\medskip

Let $\{a_1,a_2,\ldots,a_N\}$ be the zero set of $f$. Then
\begin{equation}\label{main1}
\frac{S_{\varphi_k}(z,w)}{S_{\varphi_k}(z,b)S_{\varphi_k}(b,w)} 
=
\frac{1}{1-f(z)\overline{f(w)}} \sum_{i,j=1}^N c_{ijk}\, \frac{S_{\varphi_k}(z,a_i)}{S_{\varphi_k}(z,b)} \overline{\left(\frac{S_{\varphi_k}(w,a_j)}{S_{\varphi_k}(w,b)}\right)}  
\end{equation}
where the coefficients $c_{ijk}$ are determined by the condition $[c_{ijk}]=[S_{\varphi_k}(a_i,a_j)]^{-1}$, and
\begin{equation}\label{main2}
\frac{S(z,w)}{S(z,b)S(b,w)} 
=
\frac{1}{1-f(z)\overline{f(w)}} \sum_{i,j=1}^N c_{ij}\, \frac{S(z,a_i)}{S(z,b)} \overline{\left(\frac{S(w,a_j)}{S(w,b)}\right)} 
\end{equation}
where $[c_{ij}]=[S(a_i,a_j)]^{-1}$. Since $f$ extends meromorphically to the double, it follows from Corollary \ref{Bell2} that the functions in (\ref{main1}) and (\ref{main2}) extend meromorphically to the double $\Hat{\Omega}$. Write
\[
\frac{S_{\varphi_k}(z,w_0)}{S_{\varphi_k}(z,b)S_{\varphi_k}(b,w_0)} 
= \mathcal{F}(z) \mathcal{S}_k(z) 
\quad \text{and} \quad
\frac{S(z,w_0)}{S(z,b)S(b,w_0)} 
= \mathcal{F}(z) \mathcal{S}(z) 
\]
where $\mathcal{F}$ denotes the meromorphic extension of $1/(1-f(z)\overline{f(w_0)})$ to $\Hat{\Omega}$. The values of $\mathcal{S}_k$ and $\mathcal{S}$ on $\Omega \sqcup \partial\Omega$ can be read from (\ref{main1}) and (\ref{main2}). 
From the proof of Corollary \ref{Bell2},
\[
\mathcal{S}_k(z) = \sum_{i,j=1}^N c_{ijk}\, \overline{\left(\frac{L_{\varphi_k}(R(z),a_i)}{L_{\varphi_k}(R(z),b)}
\frac{S_{\varphi_k}(w_0,a_j)}{S_{\varphi_k}(w_0,b)}\right)}
\quad \text{and}\quad
\mathcal{S}(z) = \sum_{i,j=1}^N c_{ij}\, \overline{\left(\frac{L(R(z),a_i)}{L(R(z),b)}
\frac{S(w_0,a_j)}{S(w_0,b)}\right)}
\]
for $z\in\Tilde{\Omega} = R(\Omega)$.

\medskip

Let $z_0\in\partial\Omega$ be such that $z_0\neq w_0$. It can be read from the left side in (\ref{main1}) and (\ref{main2}) that $\mathcal{F}(z)\mathcal{S}_k(z)$ and $\mathcal{F}(z)\mathcal{S}(z)$ cannot have a pole at $z=z_0$ for all $k\geq k_1$. By the same reasoning, $\mathcal{S}_k(z)$ and $\mathcal{S}(z)$ cannot have pole at $z=z_0$ for all $k\geq k_1$. Let $\psi: U\rightarrow V\subset\mathbb{C}$ be a chart near $z_0$ with $\psi(z_0)=0$ and $w_0\notin U$.
Therefore, if $\mathcal{F}(z)$ has a pole at $z=z_0$ then there exists an integer $r \ge 1$ such that for $k\geq k_1$,
\[
\mathcal{F}\circ \psi^{-1}(z) = z^{-r}\mathcal{G}(z), \quad
\mathcal{S}_k\circ \psi^{-1}(z) = z^{r} \mathcal{T}_k(z) \quad \text{and} \quad 
\mathcal{S}\circ\psi^{-1}(z) = z^{r} \mathcal{T}(z)
\]
where $\mathcal{G}, \mathcal{T}_k$ and $\mathcal{T}$ are holomorphic functions in a neighborhood of $0$. Without loss of generality, assume it to hold on $V$. 

\medskip

From Theorems \ref{ram 2} and \ref{ram 3}, the functions $\mathcal{S}_k$ converge to $\mathcal{S}$ uniformly on some neighborhood of $z_0$ in $\Hat{\Omega}$. Without loss of generality, let the convergence be on $U$. So, $\mathcal{S}_k\circ \psi^{-1}\rightarrow\mathcal{S}\circ\psi^{-1}$ uniformly on $V$. Let $C\subset V$ be a circle centered at $0$ enclosing the disc $D$. By the maximum modulus principle, 
\[
\sup_{z\in D}\vert \mathcal{T}_k(z)\vert \leq \sup_{z\in C}\vert \mathcal{T}_k(z)\vert \leq M_1 \sup_{z\in C}\vert z^r\mathcal{T}_k(z)\vert 
=  M_1 \sup_{z\in C}\vert \mathcal{S}_k\circ\psi^{-1}(z)\vert <M_2<\infty,
\]
where $M_1, M_2$ are positive constants. By Montel's theorem, every subsequence of $\{\mathcal{T}_k\vert_D\}$ has a convergent subsequence. But the only possible limit point is $\mathcal{T}\vert_D$. Therefore, $\mathcal{T}_k\rightarrow\mathcal{T}$ uniformly on $D$. This implies that $\mathcal{F}\mathcal{S}_k\rightarrow\mathcal{F}\mathcal{S}$ uniformly on $\psi^{-1}(D)$. In particular,
\[
\lim_{k\rightarrow\infty}\frac{S_{\varphi_k}(z,w_0)}{S_{\varphi_k}(z,b)S_{\varphi_k}(b,w_0)} = \frac{S(z,w_0)}{S(z,b)S(b,w_0)}
\]
uniformly for $z\in \overline{\Omega}\cap \psi^{-1}(D)$. Note that the constants $S_{\varphi_k}(b,w_0)$, $S(b,w_0)$ and the functions $S_{\varphi_k}(z,b)$, $S(z,b)$ do not vanish on some neighborhood of $z_0$ in $\overline{\Omega}$. Shrinking $U$ if necessary, assume that they do not vanish for $z\in \overline{\Omega}\cap \psi^{-1}(D)$ and $k\geq k_1$. We therefore conclude using Theorem \ref{ram 2} that 
\begin{equation}\label{con3}
\lim_{k\rightarrow\infty}S_{\varphi_k}(z,w_0) = S(z,w_0)
\end{equation}
uniformly for $z\in\overline{\Omega}\cap \psi^{-1}(D)$. Hence, we have proved the theorem.
\end{proof}

\begin{rem} 
Since the Szeg\H{o} kernel is conjugate symmetric, we also have that for every fixed $z_0\in\partial\Omega$,
\[
\lim_{k\rightarrow\infty}S_{\varphi_k}(z_0,w) = S(z_0,w)
\]
locally uniformly on $\overline{\Omega}\setminus\{z_0\}$.
\end{rem}

\medskip

We believe that $S_{\varphi_k}(z, w)$ converges to $S(z, w)$ locally uniformly on $(\overline \Omega \times \overline \Omega) \setminus \Delta$, but we have not been able to show this.

\medskip

Let $\Omega\subset\mathbb{C}$ be a bounded $n$-connected domain with $C^{\infty}$ smooth boundary and $\varphi$ be a positive $C^{\infty}$ smooth function on $\overline{\Omega}$. Let $\mathcal{O}(\Omega)$ denote the space of all holomorphic functions on $\Omega$. The space
\[
A^2_{\varphi}(\Omega) = \{f\in \mathcal{O}(\Omega): \iint_{\Omega}\vert f(z)\vert^2 \varphi(z)\,dV(z) < \infty\}
\]
is a Hilbert space with respect to the inner product
\[
\langle f,g\rangle_{L^2_{\varphi}} = \iint_{\Omega} f(z) \ov{g(z)} \varphi(z) \, dV(z)
\]
where $dV$ denotes the volume Lebesgue measure on $\Omega$,
and the evaluation functional $h\mapsto h(\zeta)$ on $A^2_{\varphi}(\Omega)$ is continuous for $\zeta\in\Omega$ (see \cite{pw1, pw2}). Hence, there exists a unique function $K_{\varphi}(\cdot,\zeta)\in A^2_{\varphi}(\Omega)$ such that 
\[
f(\zeta) = \langle f, K_{\varphi}(\cdot,\zeta)\rangle_{L^2_{\varphi}} \quad \text{ for all }f\in A^2_{\varphi}(\Omega).
\]
The space $A^2_{\varphi}(\Omega)$ is called the weighted Bergman space and the function $K_{\varphi}(\cdot,\cdot)$ is called the weighted Bergman kernel of $\Omega$ with respect to the weight $\varphi$. When $\varphi \equiv 1$, we get the classical Bergman kernel.

\medskip

Let $\gamma_j$, $j=1,\ldots,n$ denote the $n$ boundary curves of $\Omega$. The harmonic measure functions $\omega_j$ are unique harmonic functions on $\Omega$ that take value $1$ on $\gamma_j$ and $0$ on $\gamma_i$ for $i\neq j$. Let $F_j'$ denote the holomorphic functions on $\Omega$ given by $(1/2)(\partial/\partial z)\omega_j(z)$.
It is known in the classical case that (see \cite{Be0})
\begin{equation}
K(z,w) = 4\pi S(z,w)^2 + \sum_{j,k=1}^{n-1} c_{jk} F_j'(z) \ov{F_k'(w)}
\end{equation}
for some constants $c_{jk}$.

\begin{cor}
Let $\Omega\subset\mathbb{C}$ be a bounded $n$-connected domain with $C^{\infty}$ smooth boundary and $\varphi_k$ be a sequence of positive $C^{\infty}$ smooth functions on $\ov\Omega$ such that $\varphi_k\rightarrow \mathbf{1}$ uniformly on $\ov\Omega$ as $k\rightarrow\infty$. Define
\[
E_k(z,w) = K_{\varphi_k}(z,w) - 4\pi S_{\varphi_k}(z,w)^2, \quad z,w\in\Omega.
\]
Then,
\[
\lim_{k\rightarrow\infty} E_k(z,w) = \sum_{j,k=1}^{n-1} c_{jk} F_j'(z) \ov{F_k'(w)}
\]
locally uniformly on $\Omega\times\Omega$.
\end{cor}

\begin{proof}
We can write
\begin{multline*}
K_{\varphi_k}(z,w) - 4\pi S_{\varphi_k}(z,w)^2 = (K_{\varphi_k}(z,w) - K(z,w)) - 4\pi (S_{\varphi_k}(z,w)^2 - S(z,w)^2)\\ 
+ K(z,w) - 4\pi S(z,w)^2  \\
=
(K_{\varphi_k}(z,w) - K(z,w)) - 4\pi (S_{\varphi_k}(z,w)^2 - S(z,w)^2)
+ \sum_{j,k=1}^{n-1} c_{jk} F_j'(z) \ov{F_k'(w)} 
\end{multline*}
Moreover, 
\[
\lim_{k\rightarrow\infty}K_{\varphi_k}(z,w) = K(z,w)
\quad \text{and} \quad
\lim_{k\rightarrow\infty}S_{\varphi_k}(z,w) = S(z,w)
\]
locally uniformly on $\Omega\times\Omega$. The convergence of weighted Szeg\H{o} kernels follows from Theorem \ref{ram 1}. Based on the ideas from \cite{ra1, ra2}, we will give a quick proof for the convergence of weighted Bergman kernels. 

\medskip

Since $\varphi_k\rightarrow \mathbf{1}$ uniformly on $\ov D$ as $k\rightarrow\infty$, there exists a constant $c>0$ such that $\varphi_k(z)\geq c$ for all $z\in \ov{D}$. Let $K_c$ denote the Bergman kernel of $D$ with respect to the constant weight $c$. The monotonicity property of the Bergman kernel gives that $K_{\varphi_k}(z,z)\leq K_{c}(z,z)$ for all $z\in D$. 
Let $W\subset D$ be compact. Using Schwarz inequality, we have for all $z,w\in W$
\begin{eqnarray*}
\vert K_{\varphi_k}(z,w)\vert 
&\leq&
\sqrt{K_{\varphi_k}(z,z)} \sqrt{K_{\varphi_k}(w,w)}
\leq
\sqrt{K_{c}(z,z)} \sqrt{K_{c}(w,w)}\\
&=&
\frac{1}{c}\sqrt{K(z,z)} \sqrt{K(w,w)} \leq M <\infty.   
\end{eqnarray*}
Thus, $\{K_{\varphi_k}\}_{k=1}^{\infty}$ is a Montel family on $D\times D$. It suffices to show that each convergent subsequence converges to $K$. Without loss of generality, let us assume that the entire sequence $\{K_{\varphi_k}\}_{k=1}^{\infty}$ converges to a function $k_0$. The limit $k_0$ is holomorphic in the first variable and antiholomorphic in the second variable. Furthermore, for each $w\in D$
\begin{eqnarray*}
\int_{D}\vert k_0(z,w)\vert^2 dV(z) 
&\leq&
\liminf_{k\rightarrow\infty}\int_D \vert K_{\varphi_k}(z,w)\vert^2\varphi_k(z) dV(z) \\
&=&
\liminf_{k\rightarrow\infty}K_{\varphi_k}(w,w) = k_0(w,w)<\infty.
\end{eqnarray*}
Therefore, $k_0(\cdot,w)\in L^2(D)\cap \mathcal{O}(D)$. A final application of DCT gives the reproducing property
\[
f(w) = \int_{D} f(z) \ov{k_0(z,w)}dV(z)
\]
for all $f\in L^2(D)\cap \mathcal{O}(D)$. By the uniqueness of the Bergman kernel, we have $k_0 = K$.
\end{proof}

\section{Two Applications}

In this section, we provide two additional examples of how information pertaining to $S(z, w)$ can be transferred to $S_{\varphi}(z, w)$ for $\varphi$ close to the constant weight $\mathbf 1$. The first pertains to the zeros of the weighted Szeg\H{o} kernel, while the second one is about a description of certain subspaces of $L^2_{\varphi}(\partial \Omega)$ that are orthogonal to both $H^2_{\varphi}(\partial \Omega)$ and its conjugate $\overline{H^2_{\varphi}(\partial \Omega)}$ -- this is motivated by Theorem $19.1$ in \cite{Be0}.

\begin{thm}\label{Zero 1}
Let $\Omega\subset\mathbb{C}$ be a bounded $n$-connected domain with $C^{\infty}$ smooth boundary curves. Let $\varphi_k$ be a sequence of positive real-valued $C^{\infty}$ functions on $\partial\Omega$ such that $\varphi_k\rightarrow \mathbf{1}$ in $C^{\infty}$ topology on $\partial\Omega$ as $k\rightarrow\infty$. Let $a\in\Omega$. Then there exists a $k_0 \ge 1$ such that
\begin{enumerate}
    \item $S_{\varphi_k}(\cdot,a)$ has $n-1$ zeroes counting multiplicities in $\Omega$,
    \item $S_{\varphi_k}(\cdot,a)$ does not vanish on $\partial\Omega$, and
    \item $L_{\varphi_k}(\cdot,a)$ does not vanish on $\overline{\Omega}$.
\end{enumerate}
for all $k\geq k_0$. Furthermore, if $a \in \Omega$ is such that $S(\cdot,a)$ has $n-1$ simple zeroes in $\Omega$, then $k_0$ can be chosen so that $S_{\varphi_k}(\cdot,a)$ has simple zeroes in $\Omega$ for all $k\geq k_0$.
\end{thm}

\begin{proof}
Assume that the zero set of $S(\cdot, a)$ is $\{a_1,\ldots,a_{r}\}$ such that $S(z,a_i)$ vanishes at $z=a_i$ with multiplicity $m_i$. Here, we must have $m_1 + m_2 + \ldots + m_r = n-1$. Choose $\epsilon>0$ small enough that $\overline{B(a_i,\epsilon)}\subset\Omega$ for all $i$ and $\overline{B(a_j,\epsilon)}\cap\overline{B(a_i,\epsilon)}=\emptyset$ for all $j\neq i$.  

\medskip

Since $S_{\varphi_k}(\cdot,a)\rightarrow S(\cdot,a)$ uniformly on $C(a_j,\epsilon) = \partial B(a_j, \epsilon)$ and $S(\cdot,a)$ does not vanish on $C(a_j,\epsilon)$, there exists $k_1\in \mathbb{Z}^+$ such that $S_{\varphi_k}(\cdot,a)$ does not vanish on $C(a_j,\epsilon)$ for all $1\leq j\leq r$ and $k\geq k_1$. Fix $i\in \{1,\ldots,r\}$. Then,
\[
N(k) = \frac{1}{2\pi i} \int_{C(a_i,\epsilon)}\frac{(\partial/\partial z)S_{\varphi_k}(z,a)}{S_{\varphi_k}(z,a)} dz,\quad k\geq k_1,
\]
which equals the number of zeroes of $S_{\varphi_k}(\cdot,a)$ in $B(a_i,\epsilon)$ counting multiplicities, converges to 
\[
\frac{1}{2\pi i} \int_{C(a_i,\epsilon)}\frac{(\partial/\partial z)S(z,a)}{S(z,a)} dz,
\]
that gives the number of zeroes of $S(\cdot,a)$ in $B(a_i,\epsilon)$ counting multiplicities, which is $m_i$. Therefore, $N(k)$ is an eventually constant sequence, equal to $m_i$. Therefore, $S_{\varphi_k}(\cdot,a)$ has $m_i$ zeroes counting multiplicities in $B(a_i,\epsilon)$ for large enough $k$. 

\medskip 

Thus, there exists $k_2\geq k_1$ such that $S_{\varphi_k}(\cdot,a)$ has $m_j$ zeroes counting multiplicities in $B(a_j,\epsilon)$ for all $k\geq k_2$ and $1\leq j\leq r$. That is, $S_{\varphi_k}(\cdot,a)$ has at least $n-1$ zeroes counting multiplicities in $\Omega$ for all $k\geq k_2$. 

\medskip

Since $S_{\varphi_k}(\cdot,a)\rightarrow S(\cdot,a)$ uniformly on $\partial\Omega$ and $S(\cdot,a)$ does not vanish on $\partial\Omega$, there exists $k_3\geq k_2$ such that $S_{\varphi_k}(\cdot,a)$ does not vanish on $\partial\Omega$ for all $k\geq k_3$. Since
\begin{equation}\label{boundary relation}
\varphi_k(z)\,\overline{S_{\varphi_k}(z,a)}=\frac{1}{i} \,L_{\varphi_k}(z,a)\, T(z),\quad z\in\partial\Omega,
\end{equation}
the Garabedian kernels $L_{\varphi_k}(\cdot,a)$ also do not vanish on $\partial\Omega$ for all $k\geq k_3$. Therefore, from (\ref{boundary relation}) we get that
\[
\frac{1}{i} S_{\varphi_k}(z,a) \,L_{\varphi_k}(z,a)\, T(z) = \varphi_k(z) \vert S_{\varphi_k}(z,a)\vert^2 > 0 
\]
for all $z\in \partial\Omega$ and $k\geq k_3$.
Thus, $\Delta \arg (S_{\varphi_k}(\cdot,a)\,L_{\varphi_k}(\cdot,a)) + \Delta \arg T =0$. By the argument principle, this means that
\begin{multline}
2\pi \left( \text{no. of zeroes of } S_{\varphi_k}(\cdot,a)\,L_{\varphi_k}(\cdot,a) \text{ in }\Omega\, 
-
\text{ no. of poles of } S_{\varphi_k}(\cdot,a)\,L_{\varphi_k}(\cdot,a) \text{ in } \Omega \right) \\
+ 2\pi(1 - (n-1)) =0,
\end{multline}
where the zeroes and poles are counted with multiplicities. The Szeg\H{o} kernel $S_{\varphi_k}(\cdot,a)$ is holomorphic on $\Omega$ and the Garabedian kernel is holomorphic on $\Omega\setminus\{a\}$ with a simple pole at $z=a$. Therefore, the combined number of zeroes of $S_{\varphi_k}(\cdot,a)$ and $L_{\varphi_k}(\cdot,a)$ in $\Omega$ counting multiplicities is $1-(1-(n-1))=n-1$ for all $k\geq k_3$.

\medskip

We have shown the existence of at least $n-1$ zeroes of $S_{\varphi_k}(\cdot,a)$ counting multiplicities for all $k\geq k_3$. Therefore, the above statement implies that these are the only zeroes of $S_{\varphi_k}(\cdot,a)$, and $L_{\varphi_k}(\cdot, a)$ does not vanish in $\Omega$ for all $k\geq k_3$. 
\end{proof}

For a positive $C^{\infty}$ function $\rho$ on $\partial\Omega$, define $\mathcal{Q}_{\rho}$ to be the space of functions in $L^2_{\rho}(\partial\Omega)$ that are orthogonal to both the Hardy space $H^2_{\rho}(\partial\Omega)$ and to the space of functions that are complex conjugates of functions in $H^2_{\rho}(\partial\Omega)$. 

\medskip

Recall that the harmonic measure functions $\omega_j$ are unique harmonic functions on $D$ that take value $1$ on $\gamma_j$ and $0$ on $\gamma_i$ for $i\neq j$, and $F_j'$ are the holomorphic functions on $D$ given by $(1/2)(\partial/\partial z)\omega_j(z)$.

\medskip

It is known (see \cite{Be0}) that $\mathcal{Q}=\{hT : h\in \mathcal{F}'\}$, where $\mathcal{F}'=\text{span}_{\mathbb{C}}\{F_j' : j=1,\ldots, n-1\}$. Therefore, it follows immediately that
\[
\mathcal{Q}_{\varphi} = \{\varphi^{-1}hT: h\in\mathcal{F}'\}.
\]
Choose $a\in\Omega$ close to $\partial\Omega$ so that $S(\cdot,a)$ has simple zeroes $\{a_1,\ldots,a_n\}$. By \cite{Be0},
\begin{equation}
\mathcal{F}'=\text{span}_{\mathbb{C}}\{L(\cdot,a_j)S(\cdot,a) : j=1,\ldots,n-1\} = \text{span}_{\mathbb{C}}\{L(\cdot,a)S(\cdot,a_j) : j=1,\ldots,n-1\}. 
\end{equation}
The following theorem describes $\mathcal{F}'$ in terms of the weighted Szeg\H{o} and Garabedian kernels.

\begin{thm}
Let $\Omega$ be a bounded $n$-connected domain with $C^{\infty}$ smooth boundary and $\varphi_k$ a sequence of positive real-valued $C^{\infty}$ functions on $\partial\Omega$ such that $\varphi_k\rightarrow \mathbf{1}$ in the $C^{\infty}$-topology on $\partial\Omega$ as $k\rightarrow\infty$. Choose $a\in\Omega$ so that $S(\cdot,a)$ has $n-1$ distinct simple zeroes. Then there exists $k_0\ge 1$ such that $S_{\varphi_k}(\cdot,a)$ has $n-1$ distinct simple zeroes for all $k\geq k_0$ and the following holds:

\medskip

Choose $\varphi$ from the set $\{\varphi_k :k\geq k_0\}$ and let  $Z(S_{\varphi}(\cdot,a))=\{b_1,\ldots,b_{n-1}\}$ be the zero set of $S_{\varphi}(\cdot, a)$. Then, 
\begin{equation}
\mathcal{F}'
=
\text{span}_{\mathbb{C}}\{L_{\varphi}(\cdot,b_j)S_{\varphi}(\cdot,a) : j=1,\ldots,n-1\}
=
\text{span}_{\mathbb{C}}\{L_{\varphi}(\cdot,a)S_{\varphi}(\cdot,b_j) : j=1,\ldots,n-1\}. 
\end{equation}
The space $\mathcal{Q}_{\varphi}$ of functions in $L^2_{\varphi}(\partial\Omega)$ orthogonal to both $H^2_{\varphi}(\partial\Omega)$ and $\overline{H^2_{\varphi}(\partial\Omega)}$ is equal to $\varphi^{-1}\mathcal{F}'T$.
\end{thm}

\begin{proof}
Let $\mathcal{L}_{\varphi}=\text{span}_{\mathbb{C}}\{L_{\varphi}(\cdot,b_j)S_{\varphi}(\cdot,a) : j=1,\ldots,n-1\}$. We first show that $\mathcal{Q}\subset \mathcal{L}_{\varphi}T$. Let $f\in\mathcal{Q}$, i.e., $f$ is orthogonal to both $H^2(\partial\Omega)$ and $\overline{H^2(\partial\Omega)}$. Therefore, there exists $h,\,H\in H^2(\partial\Omega)$ such that 
\begin{equation}\label{orthogonal}
f= hT =\overline{HT}.
\end{equation}
Recall the identity
\begin{equation}\label{BdyRelation}
\varphi(z)\,\overline{S_{\varphi}(z,a)}=\frac{1}{i} \,L_{\varphi}(z,a)\, T(z),\quad z\in\partial\Omega.
\end{equation}
Therefore, 
\[
h \frac{i}{\varphi}\frac{\overline{L_{\varphi}(\cdot,a)}}{S_{\varphi}(\cdot,a)}= \overline{HT},
\]
which can be rewritten as 
\begin{equation}\label{orthogonal-01}
\frac{i h}{S_{\varphi}(\cdot,a)}= \varphi\,\overline{\left(\frac{H}{L_{\varphi}(\cdot,a)} T\right)}.
\end{equation}
The function on the right hand side of the above equation is orthogonal to $H^2_{1/\varphi}(\partial\Omega)$. Therefore, taking the orthogonal weighted Szeg\H{o} projection $P^{\perp}_{1/\varphi}$ with respect to the weight $1/\varphi$ on both sides of (\ref{orthogonal-01}) gives
\begin{equation}\label{Orthogonal1}
P^{\perp}_{1/\varphi} \left(\frac{i h}{S_{\varphi}(\cdot,a)}\right) = \varphi\,\overline{\left(\frac{H}{L_{\varphi}(\cdot,a)} T\right)}.
\end{equation}
We can also write
\[
\frac{i h(z)}{S_{\varphi}(z,a)}= G(z) + \sum_{j=1}^{n-1}c_j \frac{1}{z-b_j},
\]
where $G\in H^2(\partial\Omega)$ and $c_j=i h(b_j)/ {S_{\varphi}}'(b_j,a)$ is the residue of the simple pole of the meromorphic function $i h(z)/S_{\varphi}(z,a)$ at $z=b_j$. Thus,
\begin{equation}\label{Orthogonal2}
P^{\perp}_{1/\varphi} \left(\frac{i h}{S_{\varphi}(\cdot,a)}\right) = \sum_{j=1}^{n-1} c_j P^{\perp}_{1/\varphi}\left(\frac{1}{z-b_j}\right) 
\end{equation}
\medskip

It is immediate from (\ref{BdyRelation}) that $L_{\varphi}(\cdot,a)$ is orthogonal to $H^2_{1/\varphi}(\partial\Omega)$. For every $w\in\Omega$, there exists a function $H_w\in H^2(\partial\Omega)$ such that
\[
L_{\varphi}(z,w) = \frac{1}{2\pi}\frac{1}{z-w} - i H_w(z).
\]
Therefore, for $w\in\Omega$
\begin{equation}
L_{\varphi}(z,w) = P^{\perp}_{1/\varphi} (L_{\varphi}(z,w)) = \frac{1}{2\pi} P^{\perp}_{1/\varphi} \left(\frac{1}{z-w}\right).
\end{equation}
Hence, 
\begin{equation}\label{Orthogonal3}
P^{\perp}_{1/\varphi} \left(\frac{i h}{S_{\varphi}(\cdot,a)}\right) = \sum_{j=1}^{n-1} c_j P^{\perp}_{1/\varphi}\left(\frac{1}{z-b_j}\right) = 2\pi\sum_{j=1}^{n-1} c_j L_{\varphi}(\cdot,b_j).
\end{equation}
On comparing (\ref{Orthogonal1}) and (\ref{Orthogonal3}), we obtain
\begin{eqnarray*}
f
&=&
\overline{HT} 
=
\frac{1}{\varphi} \overline{L_{\varphi}(\cdot,a)} \,P^{\perp}_{1/\varphi} \left(\frac{i h}{S_{\varphi}(\cdot,a)}\right) 
=
\frac{1}{\varphi} \overline{L_{\varphi}(\cdot,a)} \,2\pi\sum_{j=1}^{n-1} c_j L_{\varphi}(\cdot,b_j)\\
&=&
-i T S_{\varphi}(\cdot,a)\,2\pi\sum_{j=1}^{n-1} c_j L_{\varphi}(\cdot,b_j)\\
&=&
\sum_{j=1}^{n-1}(-2\pi i c_j) L_{\varphi}(\cdot,b_j)S_{\varphi}(\cdot,a) T.
\end{eqnarray*}
Therefore, $\mathcal{Q}\subset \mathcal{L}_{\varphi}T$. The complex vector space $\mathcal{L}_{\varphi}T$ has dimension less than or equal to $(n-1)$ and the vector space $\mathcal{Q}$ has dimension $(n-1)$ (see \cite{Be0}). Hence, $\mathcal{Q}=\mathcal{L}_{\varphi}T$. Since $\mathcal{Q}=\mathcal{F}'T$ and 
\[
L_{\varphi}(z,\xi) S_{\varphi}(z,\zeta) T(z)= -\overline{S_{\varphi}(z,\xi) L_{\varphi}(z,\zeta) T(z)}\quad \text{for all }\xi, \zeta \text{ in }\Omega,
\]
we finally conclude that
\[
\mathcal{F}'=\text{span}_{\mathbb{C}}\{L_{\varphi}(\cdot,b_j)S_{\varphi}(\cdot,a) : j=1,\ldots,n-1\} = \text{span}_{\mathbb{C}}\{L_{\varphi}(\cdot,a)S_{\varphi}(\cdot,b_j) : j=1,\ldots,n-1\}. 
\]
\end{proof}

\section{The reduced Bergman kernel}\label{Reduced_Bergman}

Let $\Omega\subset\mathbb{C}$ be a domain, fix $\zeta\in \Omega$ and an integer $n \ge 1$. Then
\[
AD(\Omega,\zeta^n)=\lbrace f\in \mathcal{O}(\Omega): f(\zeta)=f'(\zeta)=\cdots = f^{(n-1)}(\zeta)=0\,\text{ and }\,\int_{\Omega}\lvert f'(z)\rvert^2 dxdy <\infty\rbrace.
\]
is a Hilbert space with respect to the inner product
\[
\langle f,g\rangle_{AD(\Omega,\zeta^n)}=\int_{\Omega}f'(z)\,\overline{g'(z)}\,dx dy,\quad f,g\in AD(\Omega,\zeta^n).
\]
Further, the Cauchy integral formula for the $n$-th derivative $f^{(n)}$ shows that the linear functional defined by $AD(\Omega,\zeta^n)\ni f\mapsto f^{(n)}(\zeta)\in\mathbb{C}$ is continuous. Thus, there exists a unique function $M(\cdot,\zeta^n,\Omega)\in AD(\Omega,\zeta^n)$ such that
$f^{(n)}(\zeta)=\langle f, M(\cdot,\zeta^n,\Omega)\rangle$ for every $f\in AD(\Omega,\zeta^n)$. Define
\[
\tilde{K}_{\Omega,n}(z,\zeta)=\frac{\partial}{\partial z} M(z,\zeta^n,\Omega),\quad z,\zeta\in \Omega.
\]
The kernel $\tilde{K}_{\Omega,n}$ is called the $n^{th}$-order reduced Bergman kernel of $\Omega$. So,
\begin{equation}
f^{(n)}(\zeta)=\int_{\Omega} f'(z)\,\overline{\tilde{K}_{\Omega,n}(z,\zeta)}\,dxdy
\quad
\text{for } f\in AD(\Omega,\zeta^n).
\end{equation}
For $n=1$, this gives the reduced Bergman kernel $\tilde{K}_{\Omega}$ of $\Omega$.

\begin{thm} [See \cite{Bergman}, \cite{Burbea}]
For a domain $\Omega\subset\mathbb{C}$ and $n\geq 2$, 
\begin{equation}\label{det}
\tilde{K}_{\Omega,n}(z,\zeta)=
\frac{(-1)^{n-1}}{J_{n-2}}
\det\left(
\begin{matrix}
\tilde{K}_{0,\bar{0}}(z,\zeta)& \ldots & \tilde{K}_{0,\overline{n-1}}(z,\zeta)\\
\tilde{K}_{0,\bar{0}}& \ldots & \tilde{K}_{0,\overline{n-1}}\\
\tilde{K}_{1,\bar{0}}& \ldots & \tilde{K}_{1,\overline{n-1}}\\
\vdots & &\vdots\\
\tilde{K}_{n-2,\bar{0}}& \ldots & \tilde{K}_{n-2,\overline{n-1}}
\end{matrix}
\right),
\end{equation}
where $J_{n}=\det\left(\tilde{K}_{j\bar{k}}\right)_{j,k=0}^n$ and 
\[
\tilde{K}_{j\bar{k}}(z,\zeta)=\frac{\partial^{j+k}}{\partial z^j\partial\bar{\zeta}^k}\tilde{K}_{\Omega}(z,\zeta),\quad \tilde{K}_{j\bar{k}}\equiv \tilde{K}_{j\bar{k}}(\zeta,\zeta).
\]
Here, $J_n>0$ for all $\zeta\in N_{\Omega}=\{z\in \Omega: \tilde{K}_{\Omega}(z,z)=0\}$. Thus, $\tilde{K}_{\Omega,n}\in C^{\infty}(\Omega\times (\Omega\setminus N_{\Omega}))$. If $\Omega$ is bounded, then $N_{\Omega}=\emptyset$. Therefore, $\tilde{K}_{\Omega,n}\in C^{\infty}(\Omega\times \Omega)$ when $\Omega$ is bounded.
\end{thm}

\medskip

Following Bell (\cite{Be1}), define classes of functions $\mathcal{A}, \mathcal B$ as follows. The class $\mathcal A$ is a subclass of meromorphic functions on $\Omega$ that consists of: 
\begin{enumerate}
    \item[(i)] $F_j'(z)$, $1 \le j \le n-1$, 
    \item[(ii)] $G_z(z,a)$ for a fixed point $a\in\overline{\Omega}$; here, $G$ is the classical Green's function on $\Omega$,
    \item[(iii)] $D_a G_z(z,a)$ where $D_a$ denotes a differential operator of the form $\frac{\partial^n}{\partial a^n}$ or $\frac{\partial^n}{\partial \bar{a}^n}$, and $a$ is a fixed point in $\overline{\Omega}$,
    \item[(iv)] $S_{\Omega}(z,a_1)\cdot S_{\Omega}(z,a_2)$ for fixed points $a_1,a_2\in\Omega$, and
    \item[(v)] linear combinations of functions above
\end{enumerate}

On the other hand, the class $\mathcal B$ is again a subclass of meromorphic functions on $\Omega$ and consists of:

\begin{enumerate}
    \item[(i)] $S_{\Omega}(z,a)$ or $L_{\Omega}(z,a)$ for fixed points $a\in\overline{\Omega}$,
    \item[(ii)] $\frac{\partial^m}{\partial\overline{a}^m}S_{\Omega}(z,a)$ or $\frac{\partial^m}{\partial\overline{a}^m}L_{\Omega}(z,a)$ for fixed points $a\in\overline{\Omega}$ and $m \ge 1$, and
    \item[(iii)] linear combinations of functions above.
\end{enumerate}

\begin{thm}\label{reduced kernel}
Let $\Omega\subset\mathbb{C}$ be a bounded $n$-connected domain with $C^{\infty}$ smooth boundary. Let $G_1$ and $G_2$ denote any two meromorphic functions on $\Omega$ that extend to the double of $\Omega$ to form a primitive pair, and let $A$ denote any function from the class $\mathcal{A}$ other than the zero function. The reduced Bergman kernel of $\Omega$ can be expressed as 
\[
\tilde{K}_{\Omega}(z,w)=
A(z)\overline{A(w)} \mathcal{R}(G_1(z),G_2(z),\overline{G_1(w)},\overline{G_2(w)})
\]
where $\mathcal{R}$ is a complex rational function of four complex variables.
\end{thm}

\begin{proof}
Let $K_{\Omega}$ denote the Bergman kernel of $\Omega$. It is known that (see \cite{Bergman}), there exist constants $c_{ij}$, $1\leq i,j\leq n-1$ such that for $z,w\in\Omega$ 
\begin{equation}
\tilde{K}_{\Omega}(z,w) = K_{\Omega}(z,w) + \sum_{i,j=1}^{n-1} c_{ij} F_i'(z) \overline{F_j'(w)}.
\end{equation}
It follows from \cite{Be1} that for every $1\leq j\leq n-1$
\[
F_j'(z)=A(z)R_{j}(G_1(z),G_2(z))
\quad\text{and}\quad
K_{\Omega}(z,w)=A(z)\overline{A(w)} R(G_1(z),G_2(z),\overline{G_1(w)},\overline{G_2(w)})
\]
where $R_j$, $1\leq j\leq n-1$ are complex rational functions of two complex variables and $R$ is a complex rational function of four complex variables. Therefore, there exists a complex rational function $\mathcal{R}$ of four complex variables such that 
\begin{equation}
\tilde{K}_{\Omega}(z,w)=
A(z)\overline{A(w)} \mathcal{R}(G_1(z),G_2(z),\overline{G_1(w)},\overline{G_2(w)}).
\end{equation}
for $z, w \in \Omega$.
\end{proof}

\begin{cor}
Let $\Omega\subset\mathbb{C}$ be a bounded $n$-connected domain with $C^{\infty}$ smooth boundary. There exist points $A_1,A_2,A_3\in\Omega$ such that the reduced Bergman kernel $\tilde{K}_{\Omega}(z,w)$ is a rational combination of $S_{\Omega}(z,A_1),S_{\Omega}(z,A_2)$ and $S_{\Omega}(z,A_3)$, and the conjugates of $S_{\Omega}(w,A_1),S_{\Omega}(w,A_2)$ and $S_{\Omega}(w,A_3)$.
\end{cor}

\begin{proof}
It is known (see \cite{Be2}) that there exist points $A_1$, $A_2$ and $A_3$ in $\Omega$ such that
\[
\frac{S_{\Omega}(z,A_1)}{S_{\Omega}(z,A_3)}\quad\text{and}\quad\frac{S_{\Omega}(z,A_2)}{S_{\Omega}(z,A_3)}
\]
extend to the double of $\Omega$ and form a primitive pair. Choosing $S_{\Omega}(z,A_i)S_{\Omega}(z,A_j)\in\mathcal{A}$ for $i,j\in\{1,2,3\}$ in Theorem \ref{reduced kernel} gives the result.
\end{proof}

It is known that (see \cite{Be1}) if $G_1$ and $G_2$ denote any two meromorphic functions on $\Omega$ that extend to the double of $\Omega$ to form a primitive pair, and $B$ denotes any function from the class $\mathcal{B}$ other than the zero function, then the Szeg\H{o} kernel can be expressed as 
\begin{equation}
S_{\Omega}(z,w)=
B(z)\overline{B(w)} R(G_1(z),G_2(z),\overline{G_1(w)},\overline{G_2(w)})
\end{equation}
where $R$ is a complex rational function of four complex variables. Using this, we get:

\begin{thm}
Let $\Omega\subset\mathbb{C}$ be a bounded $n$-connected domain with $C^{\infty}$ smooth boundary. There exist points $A_1,A_2,A_3$ such that for $n\geq 1$, and a fixed $w \in \Omega$, the higher order reduced Bergman kernels $\tilde{K}_{\Omega,n}(z,w)$ are rational combinations $S_{\Omega}(z,A_1),S_{\Omega}(z,A_2)$ and $S_{\Omega}(z,A_3)$.
\end{thm}

\begin{proof}
Choose $A_1$, $A_2$ and $A_3$ in $\Omega$ such that
\[
\frac{S_{\Omega}(z,A_1)}{S_{\Omega}(z,A_3)}\quad\text{and}\quad\frac{S_{\Omega}(z,A_2)}{S_{\Omega}(z,A_3)}
\]
extend to the double of $\Omega$ and form a primitive pair. Therefore, it follows that
the Szeg\H{o} kernel $S_{\Omega}(z,w)$ is a rational combination of $S_{\Omega}(z,A_1),S_{\Omega}(z,A_2)$ and $S_{\Omega}(z,A_3)$, and the conjuagates of $S_{\Omega}(w,A_1),S_{\Omega}(w,A_2)$ and $S_{\Omega}(w,A_3)$. Furthermore, the functions $F_j'(z)$ are rational combinations of $S_{\Omega}(z,A_1),S_{\Omega}(z,A_2)$ and $S_{\Omega}(z,A_3)$.

\medskip

Choose $a\in\Omega$ close enough to the boundary such that $S_{\Omega}(\cdot,a)$ has simple zeroes $a_1,\ldots,a_{n-1}$ and let $f_a$ denote the Ahlfors map associated with $a\in\Omega$. Then
\[
S_{\Omega}(z,w)=\frac{1}{1-f_a(z)f_a(w)}
\left(
c_0 S_{\Omega}(z,a)\overline{S_{\Omega}(w,a)}+\sum_{i,j=1}^{n-1}c_{ij}S_{\Omega}(z,a_i)\overline{S_{\Omega}(w,a_j)}
\right)
\]
where $c_0 = 1/S(a,a)$ and $[c_{ij}]=[S(a_i,a_j)]^{-1}$ (see \cite{Be4}).
The Ahlfors map $f_a$ is a proper holomorphic map on $\Omega$ and thus extends meromorphically to the double of $\Omega$, and hence can be written as a rational combination of the primitive pair. 
For $m\in\mathbb{Z}^+$, we therefore see that $\frac{\partial^m}{\partial \ov{w}^m}S_{\Omega}(z,w)$ is rational combination of $S_{\Omega}(z,A_1),S_{\Omega}(z,A_2)$ and $S_{\Omega}(z,A_3)$ for a fixed point $w\in\Omega$. Recall that the Bergman kernel and the Szegő kernel are related by
\[
K_{\Omega}(z,w)=4\pi S_{\Omega}(z,w)^2+\sum_{i,j=1}^{n-1}A_{ij}F_i'(z)\overline{F_j'(w)}, \quad z,w\in\Omega
\]
for some constants $A_{ij}$. Hence, we conclude that $K_{\Omega,n}(z,w)=\frac{\partial^m}{\partial \ov{w}^n}K_{\Omega}(z,w)$ is a rational combination of $S_{\Omega}(z,A_1),S_{\Omega}(z,A_2)$ and $S_{\Omega}(z,A_3)$ for all $n\geq 1$ and every fixed point $w\in\Omega$.

\medskip

Now, from the relation
\[
\tilde{K}_{\Omega}(z,w) = K_{\Omega}(z,w) + \sum_{i,j=1}^{n-1} c_{ij} F_i'(z) \overline{F_j'(w)},
\]
it follows that $\frac{\partial^m}{\partial \ov{w}^m}\tilde{K}_{\Omega}(z,w)$ is also a rational combination of $S_{\Omega}(z,A_1),S_{\Omega}(z,A_2)$ and $S_{\Omega}(z,A_3)$ for all $m\geq 1$ and every fixed point $w\in\Omega$. Finally, we conclude from the determinant formula (\ref{det}) that 
$\tilde{K}_{\Omega,n}(z,w)$ is a rational combination of $S_{\Omega}(z,A_1),S_{\Omega}(z,A_2)$ and $S_{\Omega}(z,A_3)$ for all $n> 1$ and every fixed point $w\in\Omega$.
\end{proof}


\end{document}